\documentclass[10pt]{article}
\usepackage[english]{babel}
\usepackage{epsfig}
\usepackage{latexsym}
\usepackage{bm}
\usepackage[latin1]{inputenc}
\usepackage{amsopn,amstext,amscd}
\usepackage{amsfonts}
\usepackage{amssymb}
\usepackage{amsmath}
\usepackage{amsbsy}
\usepackage{theorem}
\usepackage{graphicx}
\usepackage{color}

\newtheorem{definition}{Definition}
\newtheorem{theorem}{Theorem}

\newtheorem{lemma}{Lemma}

\newtheorem{remark}{Remark}

\input{tcilatex}

\begin{document}

\title{\bfseries {\large ESTIMATION OF HARMONIC COMPONENT IN REGRESSION WITH
CYCLICALLY DEPENDENT ERRORS}}
\author{Ivanov, A.V.$^{1}$, Leonenko, N.N.$^{2},$ Ruiz-Medina, M.D.$^{3}$
and Zhurakovsky, B.M. $^{1}$}
\maketitle

\begin{abstract}
This paper deals with the estimation of hidden periodicities in a
non-linear regression model with stationary noise displaying
cyclical dependence. Consistency and asymptotic normality are
established for the least-squares estimates.
\end{abstract}

\noindent $^{1}$ National Technical University \textquotedblright Kyiv
Polytechnic Institute\textquotedblright

37 Peremogy Ave, 03056 Kyiv, Ukraine

\medskip

\noindent $^{2}$School of Mathematics, Cardiff University,\newline
Senghennydd Road, Cardiff CF24 4YH, United Kingdom

\medskip

\noindent $^{3}$Department of Statistics and Operations Research \newline
University of Granada, Campus Fuente Nueva s/n, E-18071 Granada, Spain

\bigskip


\medskip

\begin{quote}
\end{quote}

\bigskip

\section{Introduction}

\bigskip Let us consider the regression model
\begin{equation}
x(t)=g(t,\theta )+\varepsilon (t),  \label{nonregr1}
\end{equation}%
\noindent where $g(t,\theta ):\mathbb{R}\times \Theta ^{c}\rightarrow
\mathbb{R}$ is a continuous non-linear function of unknown parameter vector $%
\theta \in \Theta ,$ $\theta =(\theta ^{1},\ldots ,\theta ^{q}),$ with $%
\Theta ^{c}$ being the closure in $\mathbb{R}^{q}$ of open set $\Theta
\subset \mathbb{R}^{q}$, and $\{\varepsilon (t),\ t\in \mathbb{R}\}$ is the
random noise process defining the error term through time. Process $%
\varepsilon $ is assumed to be a zero-mean stationary process, generated by
non-linear transformation of a stationary Gaussian process $\xi $ displaying
cyclical dependence. We address the problem of the estimation of the unknown
parameter $\theta $ from the observation of random process $\{x(t),t\in
\lbrack 0,T]\}$, when $T\rightarrow \infty .$

The least-squares estimate (LSE) $\hat{\theta}_{T}$ of an unknown parameter $%
\theta \in \Theta $, obtained from the observations $\{x(t),t\in [0,T]\},$
is any random variable $\hat{\theta}_{T}\in \Theta ^{c},$ having the
property
\begin{equation}
Q_{T}(\hat{\theta}_{T})=\inf_{\tau \in \Theta ^{c}}Q_{T}(\tau ),\quad
Q_{T}(\tau )=\frac{1}{T}\int\limits_{0}^{T}[x(t)-g(t,\tau )]^{2}dt,
\label{lse}
\end{equation}%
where $\Theta ^{c}$ is the closure of $\Theta $.

Our main interest in this paper is the problem of detecting hidden
periodicities, that is, the problem of estimation of the unknown parameters
of the regression function
\begin{equation}
g(t,\theta )=\sum_{k=1}^{N}\left( A_{k}\cos \varphi _{k}t+B_{k}\sin \varphi
_{k}t\right) ,  \label{nonregr2}
\end{equation}%
where $\theta =(\theta _{1},\theta _{2},\theta _{3},\dots ,\theta
_{3N-2},\theta _{3N-1},\theta _{3N})=(A_{1},B_{1},\varphi _{1},\ldots
,A_{N},B_{N},\varphi _{N})\in \mathbb{R}^{3N},%
\;C_{k}^{2}=A_{k}^{2}+B_{k}^{2}>0,\;k=1,\ldots ,N,\quad 0\leq \underline{%
\varphi }<\varphi _{1}<\cdots <\varphi _{N}<\overline{\varphi }<\infty .$

This paper provides the asymptotic properties of the LSE of the nonlinear
regression model (\ref{nonregr1}) with regression function (\ref{nonregr2})
and cyclical dependent stationary noise. Specifically, the consistency and
the convergence to the Gaussian distribution of the LSE of the parameters
involved in the definition of the regression function (\ref{nonregr2}) are
derived in this paper.

Although in the subsequent development, we will refer to the nonlinear
regression model (\ref{nonregr1}) with regression function (\ref{nonregr2}),
and cyclical dependent stationary noise with covariance function (\ref{cov1}%
), the results given in Section 4 and 5 on linearization, and asymptotic
uniqueness, as well as on asymptotic normality hold for a more general class
of regression functions satisfying conditions {\bfseries B1}-{\bfseries B6}
formulated below. The general class of non-linear regression functions that
could be considered includes the family of functions $g$ whose matrix-valued
measure, defined for $T>0$ by,
\begin{equation}
\boldsymbol{\mu }_{T}(d\lambda )=(\mu _{T}^{jl}(d\lambda ,\theta
))_{j,l=1}^{q},\quad \mu _{T}^{jl}(d\lambda ,\theta )=\frac{%
g_{T}^{j}(\lambda ,\theta )\overline{g_{T}^{l}(\lambda ,\theta )})d\lambda }{%
\left( \int\limits_{\mathbb{R}}\left\vert g_{T}^{j}(\lambda ,\theta
)\right\vert ^{2}d\lambda \int\limits_{\mathbb{R}}\left\vert
g_{T}^{l}(\lambda ,\theta )\right\vert ^{2}d\lambda \right) ^{\frac{1}{2}}}%
,\quad j,l=1,\dots ,q,  \label{regmes1}
\end{equation}%
\begin{equation}
g_{T}^{j}(\lambda ,\theta )=\int\limits_{0}^{T}e^{it\lambda }\frac{\partial
}{\partial \theta _{j}}g(t,\theta )dt,\quad j=1,\dots ,q,\quad \lambda \in
\mathbb{R},\quad \theta \in \Theta ,  \label{regmes2}
\end{equation}%
\noindent could weakly converge as $T\rightarrow \infty$ to an atomic
spectral measure $\boldsymbol{\mu }$ with atoms $\Xi _{regr}=\{\delta
_{1},\dots ,\delta _{n}\}.$ Limit theorems for non-linear transformations of
Gaussian stationary processes are here considered. In the derivation of
these limit results, the mentioned weak-convergence to the spectral measure
associated with the regression function, and the diagram formulae are
applied. In the discrete case this phenomenon was discussed by Yajima
(1988,1991) in some other regression scheme.

Note that the classical non-central limit theorems (Taqqu, 1979, and
Dobrushin and Major, 1979) can be viewed as particular cases of the general
setting considered here, when the noise is the non-linear transformation of
a Gaussian process with the unique singular point in the spectrum: $\Xi
_{noise}=\{0\}$, and the regression function is such that $\frac{\partial }{%
\partial \theta }g(t,\theta )\equiv 1$ ($q=1).$ In this case, the regression
measure $\mu _{T} $ is the Fejer kernel, which tends to the delta-measure
with atom at zero, that is, the limit $\mu $-measure spectrum consists of
one point: $\Xi_{regr}=\{0\}.$ Nonstandard renormalizations and special
limiting distributions are required here due to the fact that $\Xi _{noise}
\cap \Xi _{regr}\neq \emptyset.$ Some limiting distributions for the case
where the two spectral point sets $\Xi _{noise}$ and $\Xi _{regr}$ are in
fact overlapped, in the discrete case, can be derived from the papers by
Taqqu (1975, 1979), Rosenblatt (1981, 1987), Viano et al (1995), Oppenheim
et al (2002), Haye (2002), Haye and Viano (2002), Haye and Phillipe (2003),
Arcones (1994, 2000). In the continuous time case, the limiting
distributions for non-empty set $\Xi _{noise}\cap \Xi_{regr}$ can be
obtained from the papers and book by Ivanov and Leonenko (1989, 2004, 2008),
and Leonenko and Taufer (2006). For the non-linear regression model with
function $g$ given by (\ref{nonregr2}), this subject will be considered in
subsequent papers. In the present paper we consider the case when $\Xi
_{noise} \cap \Xi _{regr}= \emptyset .$ This assumption as it will be seen
leads to the asymptotic normality of the LSE of parameters of model (\ref%
{nonregr2}).

During the last thirty years, a number of papers have been devoted to limit
theorems for non-linear transformations of Gaussian processes and random
fields. The pioneer results are those of Taqqu (1975, 1979) and Dobrushin
and Major (1979) for convergence to Gaussian and non-Gaussian distributions
under long range dependence in terms of Hermite expansions, and Breuer and
Major (1983), Avram and Brown (1989), Ivanov and Leonenko (1989), Avram
(1992), Avram and Fox (1992) for convergence to Gaussian limit distribution
by using diagram formulae or graphical methods. This line of research
continues to be of interest today, see Berman (1992) for $m$-dependent
approximation approach, Ho and Hsing (1997) for martingale approach, Nualart
and Pecatti (2005) (see also Pecatti and Tudor (2004)) for using Malliavin
calculus, Avram, Leonenko and Sakhno (2010) for an extension of graphical
method for random fields, to name only a few papers. The volume of Doukhan,
Oppenheim and Taqqu (2003) contains outstanding surveys of the field. In
particular, that volume discusses different definitions of short range
dependence and long range dependence of stationary processes in terms of the
autocorrelation function (the integral of the correlation function diverges)
or the spectrum (the spectral density has a singularity at zero).

Non-linear regression models with independent or weakly dependent errors
have been extensively studied (see, for example, Hannan (1973), Ivanov and
Leonenko (1989), Ivanov (1997), Skouras (2000), Polard and Radchenko (2006)
and the references therein). The first results on non-linear regression with
long-range dependence (LRD) were obtained by Robinson and Hidalgo (1997).
They established conditions for consistency of some estimates of a parameter
of non-linear regression with LRD errors in discrete time models. Important
results on asymptotic distribution of M-estimators in non-linear regression
models with discrete time and LRD property of the noise process are
presented in Koul and Baillie (2003), and Koul (1996) papers, both, for
smooth and more general score functions.

Note that as we have considered the case of continuous time regression,
where the response variables are observed over continuous time, parameters
characterizing local regularity properties, such as parameter of
intermittency, for example, in the fractional Riesz-Bessel motion model, can
be estimated in this setting (see, for example, Avram, Leonenko and Sakhno
(2010)).

The asymptotic theory of LSE in non-linear regression with LRD has been
considered by Mukhergee (2000) and Ivanov and Leonenko (2004, 2008). The
papers by Ivanov and Leonenko (2009) and Ivanov and Orlovsky (2008) discuss
the asymptotic distributions of a class of M-estimates and L$_{p}$-estimates
($1<p<2$) in the nonlinear regression model with LRD. Our paper is a
continuation of these papers.

The problem of the estimation of the parameters characterizing the
distribution of the noise is not addressed here. This will be the subject of
subsequent papers in the spirit of the paper by Ivanov and Leonenko (2008).

\section{Stationary processes with cyclical dependence}

We recall the assumptions that will be made on the Gaussian process $\xi $
generating the random noise $\varepsilon ,$ representing the time-dependent
error term in the regression model (\ref{nonregr1}). Specifically, we will
consider a stationary process $\xi $ in continuous time defined on a
complete probability space $(\Omega ,\mathbb{F},P):$
\begin{equation*}
\xi (t)=\xi (\omega ,t):\Omega \times \mathbb{R}\longrightarrow \mathbb{R},
\end{equation*}%
\noindent satisfying the following assumption.

\textbf{A1.} Random function $\xi =\{\xi (t),t\in \mathbb{R}\}$ is a
real-valued and measurable stationary mean-square continuous Gaussian
process with $\mathit{E}\xi (t)=0, $ and $\mathit{E}\xi ^{2}(t)=1.$ Its
covariance function (c.f.) is of the form:
\begin{equation}
B\left( t\right) =\mathit{E}[\xi (0)\xi (t)]=\sum_{j=0}^{\kappa
}D_{j}B_{\alpha _{j},\varkappa _{j}}\left( t\right) ,\ t\in \mathbb{R},\
\kappa \geq 0, \quad \quad \sum_{j=0}^{\kappa }D_{j}=1,\ D_{j}\geq 0,\
j=0,\dots ,\kappa ,  \label{cov1}
\end{equation}%
\noindent where%
\begin{equation*}
B_{\alpha _{j},\varkappa _{j}}\left( t\right) =\frac{\cos \left( \varkappa
_{j}t\right) }{\left( 1+t^{2}\right) ^{\alpha _{j}/2}},\quad 0\leq \varkappa
_{0}<\varkappa _{1}<...<\varkappa _{\kappa },\quad \alpha _{j}>0,\quad t\in
\mathbb{R},\quad j=0,\dots ,\kappa ,
\end{equation*}

\begin{remark}
If $\varkappa_{0}=0$ and $0<\alpha_{0}<1,$ process $\xi$ displays
long-range dependence. Otherwise, process $\xi $ is of short range
dependence.
\end{remark}

The c.f. $B(t),$ $t\in \mathbb{R},$ admits the following spectral
representation:
\begin{equation*}
B(t)=\int\limits_{\mathbb{R}}e^{i\lambda t}f(\lambda )d\lambda ,\quad t\in
\mathbb{R},
\end{equation*}%
\noindent where the spectral density (s.d.) is of the form:
\begin{equation*}
f\left( \lambda \right) =\sum_{j=0}^{\kappa }D_{j}f_{\alpha _{j},\varkappa
_{j}}\left( \lambda \right) ,\quad \lambda \in \mathbb{R},
\end{equation*}%
\noindent with, for $j=0,\dots,\kappa ,$ $f_{\alpha
_{j},\varkappa_{j}}\left( \lambda \right) $ being defined by
\begin{equation*}
f_{\alpha _{j},\varkappa _{j}}\left( \lambda \right) =\frac{c_{1}\left(
\alpha _{j}\right) }{2}\left[ K_{\frac{\alpha _{j}-1}{2}}\left( \left\vert
\lambda +\varkappa _{j}\right\vert \right) \left\vert \lambda +\varkappa
_{j}\right\vert ^{\frac{\alpha _{j}-1}{2}}+K_{\frac{\alpha_{j}-1}{2}}\left(
\left\vert \lambda -\varkappa _{j}\right\vert \right) \left\vert \lambda
-\varkappa _{j}\right\vert ^{\frac{\alpha _{j}-1}{2}}\right] ,\quad \lambda
\in \mathbb{R},
\end{equation*}%
\noindent and
\begin{equation*}
c_{1}\left( \alpha_{j} \right) =\frac{2^{\left( 1-\alpha_{j} \right) /2}}{%
\sqrt{\pi }\,\Gamma \left( \frac{\alpha_{j}}{2}\right) }.
\end{equation*}

\bigskip \noindent Here,
\begin{equation*}
K_{\nu }\left( z\right) =\frac{1}{2}\int_{0}^{\infty }s^{\nu -1}\exp \left\{
-\frac{1}{2}\left( s+\frac{1}{s}\right) z\right\} ds,\quad z\geq 0,\quad \nu
\in \mathbb{R},
\end{equation*}%
\noindent is the modified Bessel function of the third kind and order $\nu $
or McDonald's function.

The following derived identities constitute an improvement and correction of
Anh, Knopova and Leonenko (2004). Indeed, we omit some details. For a small $%
z$, the following asymptotic expansions \ are known (see, i.e., Gradshteyn
and Ruzhik (2000), formulae 8.825, 8.445 and 8.446): if $\nu \notin \mathbb{Z%
}$,

\begin{equation*}
K_{-\nu }\left( z\right) =K_{\nu }\left( z\right) =
\end{equation*}%
\begin{equation*}
=\frac{\pi }{2\sin (\pi \nu )}\left\{ \sum_{j=0}^{\infty }\frac{%
(z/2)^{2j-\nu }}{j!\Gamma (j+1-\nu )}-\sum_{j=0}^{\infty }\frac{%
(z/2)^{2j-\nu }}{j!\Gamma (j+1+\nu )}\right\} ,
\end{equation*}%
while if $\nu =\pm m,$ where $m$ is a nonegative integer,%
\begin{equation*}
K_{\nu }\left( z\right) =\frac{1}{2}\sum_{j=0}^{m-1}\frac{(-1)^{j}(m-j-1)!}{%
j!}\left( \frac{z}{2}\right) ^{2j-m}+
\end{equation*}%
\begin{equation*}
+(-1)^{m+1}\sum_{j=0}^{\infty }\frac{(z/2)^{m+2j}}{j!(m+j)!}\left\{ \ln
\frac{z}{2}-\frac{1}{2}\Psi (j+1)-\frac{1}{2}\Psi (j+m+1)\right\} ,
\end{equation*}%
where $\Psi (z)=(\frac{d}{dz}\Gamma (z))/\Gamma (z)$ is the logarithm
derivative of the Gamma function.

\noindent Therefore we have: for $\alpha _{j}>1$
\begin{equation*}
\lim_{\lambda \rightarrow 0}f_{\alpha _{j},0}\left( \lambda \right) =\frac{%
\Gamma \left(\frac{\alpha _{j}-1}{2}\right)}{\left[ 2\sqrt{\pi }\Gamma (%
\frac{\alpha _{j}}{2})\right]} ,
\end{equation*}%
for $\alpha _{j}=1,$ and $\lambda \rightarrow 0$%
\begin{equation*}
f_{1,0}\left( \lambda \right) \sim \frac{1}{\pi }\left\{ -\ln
\left\vert \lambda \right\vert +\ln 2+\Psi (1)\right\} ,
\end{equation*}%
where $\Psi (1)=-\gamma ,\gamma $ is the Euler constant.

For $0<\alpha _{j}<1,$ and $\lambda \rightarrow 0$
\begin{equation*}
f_{\alpha _{j},0}\left( \lambda \right) =c_{2}\left( \alpha _{j}\right)
\frac{1}{\left\vert \lambda \right\vert ^{1-\alpha _{j}}}(1-h_{j}(\left\vert
\lambda \right\vert )),
\end{equation*}%
\noindent where $c_{2}(\alpha _{j})=[2\Gamma (\alpha _{j})\cos \frac{\alpha
_{j}\pi }{2}]^{-1},$ and $|h_{j}(|\lambda |)|<1.$

\noindent Thus, for $j=0,\dots ,\kappa ,$ $0<\alpha _{j}<1,$ in the
neighborhood of the points $\varkappa_{j}:$
\begin{equation}
f_{\alpha _{j},\varkappa _{j}}\left( \lambda \right) =c_{2}\left( \alpha
_{j}\right) \left[ \left\vert \lambda +\varkappa _{j}\right\vert ^{\alpha
_{j}-1}\left( 1-h_{j}\left( \left\vert \lambda +\varkappa _{j}\right\vert
\right) \right)\right] .  \label{f1}
\end{equation}
Therefore, the s.d. $f$ has $2\kappa +2$ different singular points $\left\{
-\varkappa _{\kappa },-\varkappa _{\kappa -1},..,-\varkappa _{1},-\varkappa
_{0},\varkappa _{0},\varkappa _{1},...,\varkappa _{\kappa }\right\} $under
condition \textbf{A1}, when $\varkappa _{0}\neq 0,$ and $0<\alpha _{j}<1,$ $%
j=0,\dots ,\kappa .$ If $\varkappa _{0}=0,$ the s.d. $f$ has $2\kappa +1$
different singular points.

For $\alpha _{j}=1$ and $\lambda \rightarrow \pm \varkappa _{j}:$
\begin{equation*}
f_{1,\varkappa _{j}}\left( \lambda \right) \sim \frac{c_{1}\left( \alpha
_{j}\right) }{2}K_{0}\left( \left\vert 2\varkappa _{j}\right\vert \right) +%
\frac{1}{2}\frac{1}{\pi }\left\{ - \ln \left\vert \lambda \mp \varkappa
_{j}\right\vert +\ln 2+\Psi (1)\right\} ,
\end{equation*}%
while for $\alpha _{j}>1$ and $\lambda \rightarrow \pm \varkappa _{j}:$%
\begin{equation*}
f_{\alpha _{j},\varkappa _{j}}\left( \lambda \right) \rightarrow \frac{%
c_{1}\left( \alpha _{j}\right) }{2}K_{\frac{\alpha _{j}-1}{2}}\left(
\left\vert 2\varkappa _{j}\right\vert \right) +\frac{1}{2}\frac{\Gamma (%
\frac{\alpha _{j}-1}{2})}{\left[ 2\sqrt{\pi }\Gamma (\frac{\alpha _{j}}{2})%
\right] }.
\end{equation*}

Similar results can be obtained for c.f.'s defined as linear combinations of
the functions
\begin{equation*}
R_{\alpha _{j},\varkappa _{j}}\left( t\right) =\frac{\cos \left( \varkappa
_{j}t\right) }{\left( 1+\left\vert t\right\vert ^{\rho _{j}}\right) ^{\alpha
j}},\quad \varkappa _{j}\in \mathbb{R},\quad 0<\rho _{j}\leq 2,\quad \alpha
_{j}>0,\varkappa _{j}\neq 0,\quad j=0,\dots ,\kappa
\end{equation*}%
\noindent (see again Ivanov and Leonenko, 2004, and Anh, Knopova and
Leonenko, 2004, for details, also some formulae of the last paper have been
corrected in the text).

\section{Consistency}

This section is devoted to the derivation of the weak-consistency of the LSE

parameter estimator, in the Walker sense. Some additional conditions are
first formulated, needed in the subsequent results.

\noindent\ \ \ \ \ \ \ \ \textbf{A2}. The stochastic process $\varepsilon $
is given by $\varepsilon (t)=G(\xi (t)),$ $t\in \mathbb{R},$ with $\xi (t)$
satisfying condition \textbf{A1}, and $G:\mathbb{R}\longrightarrow \mathbb{R}
$ being a non-random measurable function such that $\mathit{E}G(\xi (0))=0,$
and $\mathit{E}G^{4}(\xi (0))<\infty .$

Under condition \textbf{A2}, function $G\in L_{2}(\mathbb{R},\varphi (x)dx),$
with $\varphi (x)=\frac{1}{\sqrt{2\pi }}e^{-\frac{x^{2}}{2}},$ $x\in \mathbb{%
R},$ being the standard Gaussian density, and
\begin{equation}
G(x)=\sum_{k=1}^{\infty }\frac{C_{k}}{k!}H_{k}(x),\quad \sum_{k=1}^{\infty }%
\frac{C_{k}^{2}}{k!}=\mathit{E}[G^{2}(\xi (0))]<\infty ,  \label{3.1}
\end{equation}%
\noindent where%
\begin{equation*}
C_{k}=\int_{\mathbb{R}}G(x)H_{k}(x)\varphi (x)dx.
\end{equation*}%
Here, the Hermite polynomials
\begin{equation*}
H_{k}(x)=(-1)^{k}[\sqrt{2\pi }\varphi \left( x\right) ]^{-1}\frac{d^{k}}{%
dx^{k}}[\sqrt{2\pi }\varphi \left( x\right) ],\quad k=0,1,2,\ldots ,
\end{equation*}%
\noindent constitute a complete orthogonal system in the Hilbert space $%
L_{2}(\mathbb{R},\varphi (x)dx).$

\begin{remark}
Denoting by $\phi $ the distribution function (d.f.) of the standard
normal distribution, it is easy to see that the process $\varepsilon
(t)=G(\xi (t))=F^{-1}(\phi (\xi(t)))$ has a marginal d.f. $F$ for
any strictly increasing d.f. with zero mean. Thus, we can introduce
regression models with Student errors, for example.
\end{remark}

\textbf{A3}. We assume that the function $G$ has Hermite rank $Hrank(G)=m,$
that is, either $C_{1}\neq 0$ and $m=1,$ or, for some $m\geq 2,$ $%
C_{1}=\cdots =C_{m-1}=0,\ C_{m}\neq 0.$

Under conditions \textbf{A1}-\textbf{A3}, the process $\{\varepsilon
(t)=G(\xi (t)),\ t\in \mathbb{R}\},$ admits a Hermite series expansion in
the Hilbert space $L_{2}(\Omega ,\mathbb{F} ,P):$%
\begin{equation}
\varepsilon (t)=G(\xi (t))=\sum_{k=m}^{\infty }\frac{C_{k}}{k!}H_{k}(\xi
(t)).  \label{3.3}
\end{equation}

We use the following modification of the LSE proposed by Walker (1973), see
also Ivanov (1980, 2010). Consider a monotone non-decreasing system of open
sets $S_{T}\subset S(\underline{\varphi },\overline{\varphi }),\ T>T_{0}>0,$
given by the condition that the true value of unknown parameter $\varphi ,$
belongs to $S_{T},$ and
\begin{equation}
\lim_{T\rightarrow \infty }\inf_{1\leq j<k\leq N,\ \varphi \in
S_{T}}T(\varphi _{k}-\varphi _{j})=+\infty ,\quad \lim_{T\rightarrow \infty
}\inf_{\varphi \in S_{T}}T\varphi _{1}=+\infty ,  \label{3.6}
\end{equation}%
\noindent where
\begin{equation*}
S(\underline{\varphi },\overline{\varphi })=\left\{ 0\leq \underline{\varphi
}<\varphi _{1}<\dots <\varphi _{N}<\overline{\varphi }<\infty \right\} .
\end{equation*}

\begin{remark}
Assumption (\ref{3.6}) allows to distinguish the parameters $\varphi_{k},$ $%
k=1,\dots,N,$ and prove the consistency of the LSE (see Theorem 1 below).
\end{remark}

The LSE \ $\widehat{\theta }_{T}$ in the Walker sense of unknown parameter $%
\theta =(A_{1},B_{1},\varphi _{1},\ldots ,A_{N},B_{N},\varphi _{N}$ ) in the
model (\ref{nonregr1}) with nonlinear regression function (\ref{nonregr2})
is said to be any random vector $\hat{\theta}_{T}\in \Theta ^{c}$ having the
property:%
\begin{equation}
Q_{T}(\hat{\theta}_{T})=\inf_{\tau \in \Theta ^{c}}Q_{T}(\tau ),  \label{3.4}
\end{equation}%
\noindent where $Q_{T}(\tau )$ is defined in (\ref{lse}), and $\Theta
\subset \mathbb{R}^{3N}$ is such that $A_{k}\in \mathbb{R},\ B_{k}\in
\mathbb{R},\ k=1,\ldots ,N$, and $\varphi \in S_{T}^{c},$ the closure in $%
\mathbb{R}^{N}$ of the set $S_{T}.$

\begin{remark}
The 2-nd condition (\ref{3.6}) is satisfied if $\underline{\varphi }>0.$ If $%
S_{T}\subset S(\underline{\varphi },\overline{\varphi }),$ the relations
given in (\ref{3.6}) are, for example, satisfied for a parametric set $S_{T}$%
, such that
\begin{equation*}
\inf_{1\leq j<k\leq N,\ \varphi \in S_{T}}(\varphi _{k}-\varphi
_{j})=T^{-1/2},\quad \inf_{\varphi \in S_{T}}\varphi _{1}=T^{-1/2}.
\end{equation*}
\end{remark}

\begin{theorem}
Under conditions \textbf{A1} and \textbf{A2}, the LSE in the Walker sense
\begin{equation*}
\hat{\theta}_{T}=(\hat{A}_{1T},\hat{B}_{1T},\hat{\varphi}_{1T},\ldots ,\hat{A%
}_{NT},\hat{B}_{NT},\hat{\varphi}_{NT})
\end{equation*}%
\noindent of the unknown parameter $\theta =(A_{1},B_{1},\varphi _{1},\ldots
,A_{N},B_{N},\varphi _{N})$ of the regression function (\ref{nonregr2}) is
weakly consistent as $T\rightarrow \infty $, that is,
\begin{equation*}
\hat{A}_{k_{T}}\overset{P}{\longrightarrow }A_{k},\ \hat{B}_{k_{T}}\overset{P%
}{\longrightarrow }B_{k},\ T(\hat{\varphi}_{k_{T}}-\varphi _{k})\overset{P}{%
\longrightarrow }0,\quad k=1,\ldots ,N,
\end{equation*}%
\noindent where $\overset{P}{\longrightarrow }$ stands for the convergence
in probability.
\end{theorem}

The proof of the Theorem 1 is based on the diagram technique. Let us first
introduce the main elements involved in the definition of a diagram.
Specifically, a graph $\Gamma =\Gamma $($l_{1},\ldots ,l_{p})$ with $%
l_{1}+\cdots +l_{p}$ vertices is called a diagram of order ($l_{1},\ldots
,l_{p})$ if:

\textbf{a)} the set of vertices $V$ of the graph $\Gamma $ is of the form $%
V=\cup _{j=1}^{p}W_{j}$, where $W_{j}=\{(j,l),l=1,\ldots ,l_{j}\}$ is the $j$%
the level of the graph $\Gamma $, $1\leq j\leq p$ (if $l_{j}=0,$ assume that
$W_{j}=\emptyset $);

\textbf{b)} each vertex is of degree 1;

\textbf{c)} if $((j_{1},l_{1}),\ (j_{2},l_{2}))\in \Gamma $ then $j_{1}\neq
j_{2}$, that is, the edges of the graph $\Gamma $ may connect only different
levels.

Let $\mathit{L}=\mathit{L}(l_{1},\ldots ,l_{p})$ be a set of diagrams $%
\Gamma $ of order $(l_{1},\ldots ,l_{p})$. Denote by $R(\Gamma )$ the set of
edges of a graph $\Gamma \in L$. For the edge $\varpi =((j_{1},l_{1}),\
(j_{2},l_{2}))\in R(\Gamma )$, $\;j_{1}<j_{2}$, we set $d_{1}(\varpi )=j_{1}$%
,$\;d_{2}(\varpi )=j_{2}$. We call a diagram $\Gamma $ regular if its levels
can be split into pairs in such a manner that no edge connects the levels
belonging to different pairs. We denote by $L^{\ast }$ the set of regular
diagrams $L^{\ast }\subseteq L(l_{1},\ldots ,l_{p})$, If $p$ is odd, then $%
L^{\ast }=\emptyset .$ The following lemma provides the so called Diagram
Formula (see Lemma 3.2, or Doukhan, Oppenheim and Taqqu (2003), p.74 or
Pecatti and Taqqu (2010)).

\begin{lemma}
\label{lem2} Let $(\zeta _{1},\ldots ,\zeta _{p}),\ p\geq 2$, be a Gaussian
vector with $\mathit{E}\zeta _{j}=0,\;\mathit{E}\zeta _{j}^{2}=1,\;\mathit{E}%
\zeta _{i}\zeta _{j}=B(i,j),\;i,j=1,\ldots ,p,$ and let $H_{l_{1}}(u),\ldots
,H_{l_{p}}(u)$ be the Hermite polynomials. Then,
\begin{equation}
\mathit{E}\left\{ \prod_{j=1}^{p}H_{l_{j}}(\zeta _{j})\right\} =\sum_{\Gamma
\in \mathit{L}}I_{\Gamma },  \label{3.8}
\end{equation}%
\noindent where $I_{\Gamma }=\prod_{\varpi \in R(\Gamma )}B(d_{1}(\varpi
),d_{2}(\varpi )).$
\end{lemma}

As the special case for $p=2$ we have the following%
\begin{equation}
\mathit{E}H_{k}(\zeta (t))H_{l}(\zeta (s))=k!\delta _{k}^{l}B^{k}(t-s),
\label{3.9}
\end{equation}%
where $\delta _{k}^{l}$ is the Kronecker delta.

\begin{lemma}
\label{lem22} Suppose conditions \textbf{A1} and \textbf{A2} are fulfilled.
Then
\begin{equation*}
\lim_{T\rightarrow \infty }\mathit{E}\eta ^{2}(T)=0,
\end{equation*}%
where $\eta (T)=\sup_{\lambda \in \mathbb{R}}\frac{1}{T}\left\vert
\int_{0}^{T}e^{-i\lambda t}\varepsilon (t)dt\right\vert .$
\end{lemma}

\begin{proof}
Some ideas from Ivanov (2010) are used in the development of the proof of this lemma.   First, from
\begin{equation*}
\left\vert \int_{0}^{T}e^{-i\lambda t}\varepsilon (t)dt\right\vert
^{2}=\int_{-T}^{T}e^{-i\lambda u}\int_{0}^{T-\left\vert u\right\vert
}\varepsilon (t+\left\vert u\right\vert )\varepsilon
(t)dtdu=2\int_{0}^{T}\cos \lambda u\int_{0}^{T-u}\varepsilon
(t+u)\varepsilon (t)dtdu,
\end{equation*}%
we obtain%
\begin{eqnarray}
& & E\eta ^{2}(T)\leq \frac{2}{T^{2}}\int_{0}^{T}E\left\vert \int_{0}^{T-u}\varepsilon (t+u)\varepsilon
(t)dt\right\vert du
\nonumber\\
&& \hspace*{1cm}\leq \sum_{i,j=1}^{2}\frac{2}{T^{2}}\int_{0}^{T}E\left\vert \int_{0}^{T-u}G_{i}(\xi (t+u))\ G_{j}(\xi
(t))dt\right\vert du=\sum_{i,j=1}^{2}I_{i,j}(T) \label{3.10}
\end{eqnarray}
\noindent with
 $G(x)=G_{1}(x)+G_{2}(x),$ where
$G_{1}(x)=\sum_{k=1}^{M}\frac{C_{k}}{k!}H_{k}(x),$ and
$G_{2}(x)=\sum_{k=M+1}^{\infty }\frac{C_{k}}{k!}H_{k}(x).$

 Let us now compute upper bounds for $I_{i,j}(T),$ $i,j=1,2.$ In particular, from Cauchy-Schwarz inequality:%
\begin{equation*}
I_{12}(T)\leq \left( \frac{2}{T^{2}}\int_{0}^{T}(T-u)du\right) \left( EG_{1}^{2}(\xi (0)\right) ^{1/2}\left(
EG_{2}^{2}(\xi (0)\right) ^{1/2}<\left( EG_{1}^{2}(\xi (0)\right) ^{1/2}\epsilon ,
\end{equation*}%
\noindent since from (\ref{3.1}), $EG_{2}^{2}(\xi (0))=\sum_{k=M+1}^{\infty }\frac{C_{k}^{2}}{k!}$ is the tail of a
convergent series, and similarly, $I_{21}(T)<\left( EG_{1}^{2}(\xi (0)\right) ^{1/2}\epsilon ,$ and $I_{22}(T)<\epsilon
^{2},$  for any $\epsilon >0.$

Note that
\begin{equation}
I_{11}(T)=\frac{2}{T^{2}}\int_{0}^{T}E\left\vert \int_{0}^{T-u}G_{1}(\xi (t+u))G_{1}(\xi (t))dt\right\vert du\leq
\frac{2}{T^{2}}\int_{0}^{T}\psi ^{1/2}(u)du,  \label{3.11}
\end{equation}
\noindent where
\begin{eqnarray}
\psi (u) &=&\int_{0}^{T-u}\int_{0}^{T-u}EG_{1}(\xi (t+u))G_{1}(\xi
(s+u))G_{1}(\xi (t))G_{1}(\xi (s))dtds \nonumber  \\
&=&\sum_{l_{1},l_{2},l_{3},l_{4}=1}^{M}\left( \prod\limits_{j=1}^{4}\frac{%
C_{l_{j}}}{l_{j}!}\right) \int_{0}^{T-u}\int_{0}^{T-u}E\left( \prod\limits_{j=1}^{4}H_{l_{j}}\left( \xi _{j}\right)
\right) dtds,   \label{3.12}
\end{eqnarray}%
\noindent with $\left( \xi _{1},\xi _{2},\xi _{3},\xi _{4}\right)
=\left( \xi (t+u),\xi (s+u),\xi (t),\xi (s)\right).$

Applying diagram formula (\ref{3.8}) with $p=4,$ for $\Gamma \in L^{\ast }(l_{1},l_{2},l_{3},l_{4}),$ we have different
splitting of the levels (1,2,3,4) into pairs:

$ (i)\;\;\;\;  (1,2)\;\;\; (3,4) \quad (ii)\;\;\; (1,3)\;\;\;  (2,4)
\quad (iii)\;\;  (1,4)\;\;\;  (2,3). $

Let us denote the cardinality of the levels of the first pairs in (i), (ii) and (iii) as $r(1),$ and the cardinality of
the levels of the second pairs in (i), (ii) and (iii) as $r(2);$ $r(1)$ and $r(2)$ are the orders of Hermite
polynomials in the left-hand side of (\ref{3.8}).

For the product%
\begin{equation}
\prod\limits_{\varpi \in R (\Gamma )}B(d_{1}(\varpi ), d_{2}(\varpi
)) \label{3.14}
\end{equation}%
\noindent in (\ref{3.8}), we obtain the following estimates: In case (i), product  (\ref{3.14}) is bounded by
$B^{r(1)+r(2)}(t-s),$ while, for the cases (ii) and (iii), the expression (\ref{3.14}) is bounded by
$B^{r(1)+r(2)}(u),$ and $B^{r(1)}(t-s+u)B^{r(2)}(t-s-u),$ respectively.

From (\ref{3.8}), (\ref{3.11}), (\ref{3.12}) and (\ref{3.14}), we have
\begin{equation}
I_{11}(T)\leq 2\sum_{l_{1},l_{2},l_{3},l_{4}=1}^{M}\left(
\prod\limits_{j=1}^{4}\frac{C_{l_{j}}}{l_{j}!}\right) \sum_{\Gamma \in L}%
\frac{1}{T^{2}}\int_{0}^{T}\left[ \int_{0}^{T-u}\int_{0}^{T-u}\prod \limits_{\varpi \in R (\Gamma
)}\left|B(d_{1}(\varpi ), d_{2}(\varpi ))\right|dtds\right] ^{1/2}du. \label{3.15}
\end{equation}

Since $r(1),r(2)\geq 1,$ we need to estimate for the regular diagram in (\ref{3.15}) the following integrals:
\begin{equation*}
\frac{1}{T^{2}}\int_{0}^{T}\gamma _{i}^{1/2}(u)du,\ i=1,2,3,
\end{equation*}%
\noindent where
\begin{eqnarray*}
\gamma _{1}(u) &=&\int_{0}^{T-u}\int_{0}^{T-u}B^{2}(t-s)dtds, \quad
\gamma _{2}(u)= (T-u)^{2}B^{2}(u), \\
\gamma _{3}(u) &=&\int_{0}^{T-u}\int_{0}^{T-u}|B(t-s+u)B(t-s-u)|dtds.
\end{eqnarray*}

Note that
\begin{equation*}
\left\vert B(t)\right\vert =\left\vert \sum_{j=0}^{\kappa
}D_{j}B_{\alpha
_{j},\varkappa _{j}}(t)\right\vert \leq (1+t^{2})^{-\alpha %
/2}=B_{0}(t).
\end{equation*}%
Consider first the case  $\alpha =\min (\alpha _{0},\ldots ,\alpha
_{r})\in (0,1).$ Introduce the function $L(t)=\left[
\frac{t^{2}}{1+t^{2}}\right] ^{\alpha /2}$ being  monotonically
nondecreasing slowly varying at infinity function (s.v.f), see
Seneta (1976). Then, $B_{0}(t)=\frac{L(t)}{|t|^{\alpha }}.$

Denote by $b_{u}^{0}(t-s)=B_{0}(t-s-u)B_{0}(t-s+u),$ with $$B_{0}(t)= \frac{L(t)}{\left\vert t\right\vert ^{\alpha
}}.$$

 Using the change
of variables $\left( t\mapsto \frac{t}{T},s\rightarrow
\frac{s}{T},t-s\rightarrow t\right) ,$ we obtain
\begin{eqnarray*}
\gamma _{3}(u) &\leq & T^{2}\int_{0}^{1-\frac{u}{T}}\int_{0}^{1-\frac{u}{T}%
}|b_{u}^{0}(T(t-s))|dtds=T^{2}\left( 1-\frac{u}{T}\right) \int_{-1+\frac{u}{T}%
}^{1+\frac{u}{T}}\left( 1-\frac{\left\vert t\right\vert }{1-\frac{u}{T}}%
\right) |b_{u}^{0}(Tt)|dt \\
&\leq &T^{2}\left( 1-\frac{u}{T}\right) \int_{-1}^{1}\left|b_{u}^{0}(Tt)\right|dt\leq
T^{2}\left( 1-\frac{u}{T}\right) \left[ \int_{0}^{1}|B_{0}(Tt+u)|dt+%
\int_{-1}^{0}|B_{0}(Tt-u)|dt\right] .
\end{eqnarray*}

Since
\begin{equation*}
\int_{-1}^{0}B_{0}(Tt-u)dt=\int_{0}^{1}B_{0}(Tt+u)dt,
\end{equation*}%
\noindent we rewrite
\begin{equation}
\gamma _{3}(u)\leq 2T^{2}\left( 1-\frac{u}{T}\right) \int_{0}^{1}|B_{0}(Tt+u)|dt. \label{3.16}
\end{equation}

From \textbf{A1} and the monotonicity of the function $L(t),$  for
any $\epsilon >0,$ and, for $T$ sufficiently large, the following
inequalities hold:
\begin{equation*}
B_{0}(Tt+u)=\frac{L(Tt+u)}{\left\vert Tt+u\right\vert ^{\alpha }}%
\leq \frac{L(2T)}{(Tt)^{\alpha }}<\frac{1+\epsilon }{t^{\alpha }}\frac{L(T)}{T^{\alpha }}=\frac{%
1+\epsilon }{t^{\alpha }}B_{0}(T),
\end{equation*}%
\noindent and, hence,
\begin{equation}
\frac{1}{T^{2}}\int_{0}^{T}\gamma _{3}^{1/2}(u)du\leq \frac{2\sqrt{2}}{3}%
\left( \frac{1+\epsilon }{1-\alpha }\right) ^{1/2}B_{0}^{1/2}(T).  \label{3.17}
\end{equation}%
Similarly, we have $\gamma _{1}(u)\leq 2T^{2}\left( 1-\frac{u}{T}\right) \int_{0}^{1}B_{0}^{2}(Tt)dt\leq 2T^{2}\left(
1-\frac{u}{T}\right) \int_{0}^{1}B_{0}(Tt)dt.$
Thus,%
\begin{equation}
\frac{1}{T^{2}}\int_{0}^{T}\gamma _{1}^{1/2}(u)du\leq \frac{2\sqrt{2}}{3\sqrt{1-\alpha }}%
B_{0}^{1/2}(T). \label{3.18}
\end{equation}

On the other hand,
\begin{eqnarray}
\frac{1}{T^{2}}\int_{0}^{T}\gamma _{2}^{1/2}(u)du &\leq &\frac{1}{T^{2}}%
\int_{0}^{T}(T-u)|B(u)|du  \label{3.20bb} \\
&\leq &\int_{0}^{1}(1-u)B_{0}(Tu)du\leq \frac{1}{\left( 1-\alpha \right) \left( 2-\alpha \right) }B_{0}(T). \nonumber
\end{eqnarray}%
\noindent Thus, all the terms in (\ref{3.12}), corresponding to the regular diagrams, tend to zero as $T\rightarrow
\infty .$

Let us now consider the non-regular diagrams in (\ref{3.12}). Fix $\Gamma \in L\backslash L^{\ast }.$
In the product%
\begin{equation}
\prod\limits_{\varpi \in R(\Gamma )}B(d_{1}(\varpi ),d_{2}(\varpi )), \label{3.21}
\end{equation}%
there is a multiplier $B(t-s)$ (which means that an edge between
levels 1 and 2 or 2 and 4 exists), or there is a multiplier $B(u)$
(which means that an edge between the levels 1 and 3 or 2 and 4
exists). If the diagram $\Gamma $ has no edges with such properties,
then, level 1 will be connected to level 4, and level 2 will be
connected to level 3, which is true for regular diagram only. Thus,
expression  (\ref{3.15}) is given in terms of either $B(t-s)$ or
$B(u),$ and similarly to (\ref{3.18}) and (\ref{3.20bb}) one can
obtain
\begin{equation}
\frac{1}{T^{2}}\int_{0}^{T}\left[ \int_{0}^{T-u}\int_{0}^{T-u}|B(t-s)|dtds%
\right] ^{1/2}du\leq \frac{2\sqrt{2}}{3(1-\alpha )}B_{0}^{1/2}(T),  \label{3.22}
\end{equation}%
\noindent and
\begin{equation}
\frac{1}{T^{2}}\int_{0}^{T}(T-u)|B(u)|^{1/2}du\leq \frac{4}{\left( 2-%
\alpha \right) \left( 4-\alpha \right) }B_{0}^{1/2}(T).  \label{3.23}
\end{equation}

From (\ref{3.17})-(\ref{3.23}), expression (\ref{3.15}) tends to zero when $T\longrightarrow \infty ,$ and hence, the
statement of Lemma 2 follows, for $\alpha <1.$ The case $\alpha >1$ is almost obvious because of the integrability of
the function $B_{0}(t).$ If $\alpha =1,$ integrals of $B_{0}(t)$ are of logarithmic order in $T$ and the statement of
Lemma \ref{lem22} is also true.
\end{proof}

\medskip

The proof of Theorem 1 is now derived.

\begin{proof}

Denote
$$z_{k_{T}}=\frac{\sin [T(\hat{\varphi}_{k_{T}}-\varphi _{k})]}{T(\hat{\varphi}%
_{k_{T}}-\varphi _{k})},\quad y_{k_{T}}=\frac{1-\cos [T(\hat{\varphi}%
_{k_{T}}-\varphi _{k})]}{T(\hat{\varphi}_{k_{T}}-\varphi _{k})}.$$

We shall show that for $k=1,2,\ldots ,N,$%
\begin{eqnarray}
\hat{A}_{k_{T}} &=& A_{k}z_{k_{T}}-B_{k}y_{k_{T}}+o_{P}(1),  \label{3.24} \\
\hat{B}_{k_{T}} &=&A_{k}y_{k_{T}}+B_{k}z_{k_{T}}+o_{P}(1),
\nonumber
\end{eqnarray}%
where $o_{P}(1)$ means  (different) stochastic processes  tending to zero in probability as $T\longrightarrow \infty.$

Taking derivatives of the functional $Q_{T}(\theta )$ with respect to $%
A_{k},B_{k},k=1,2,\ldots ,N,$ we obtain the following system of
linear equations in terms of the  LSE $\hat{A}_{k_{T}},\hat{B}_{k_{T}},k=1,2,\ldots ,N:$%
\begin{equation}
\left\{
\begin{array}{c}
\sum_{k=1}^{N}a_{kj}^{(1)}(T)\hat{A}_{k_{T}}+\sum_{k=1}^{N}b_{kj}^{(1)}(T)%
\hat{B}_{k_{T}}=c_{j}^{(1)}(T), \\
\sum_{k=1}^{N}a_{kj}^{(2)}(T)\hat{A}_{k_{T}}+\sum_{k=1}^{N}b_{kj}^{(2)}(T)%
\hat{B}_{k_{T}}=c_{j}^{(2)}(T)%
\end{array}%
\right. ,j=1,\ldots ,N,  \label{3.25}
\end{equation}%
where, denoting $$<u(t),v(t)>=\frac{1}{T}\int_{0}^{T}u(t)v(t)dt,$$
\noindent for $k,j = 1,\ldots ,N,$
\begin{eqnarray*}
&&%
\begin{array}{cc}
a_{kj}^{(1)}(T)=\left\langle \cos (\hat{\varphi}_{k_{T}}t),\cos (\hat{\varphi}_{j_{T}}t)\right\rangle, \; &
a_{kj}^{(2)}(T)=\left\langle \cos (\hat{\varphi}_{k_{T}}t),\sin (\hat{\varphi}_{j_{T}}t)\right\rangle, \\
b_{kj}^{(1)}(T)=\left\langle \sin (\hat{\varphi}_{k_{T}}t),\cos (\hat{\varphi}%
_{j_{T}}t)\right\rangle, \; & b_{kj}^{(2)}(T)=\left\langle \sin (\hat{\varphi}%
_{k_{T}}t),\sin (\hat{\varphi}_{j_{T}}t)\right\rangle, \\
c_{j}^{(1)}(T)=\left\langle x(t),\cos (\hat{\varphi}_{j_{T}}t)\right\rangle,
\;\;\;\;\;\; & c_{j}^{(2)}(T)=\left\langle x(t),\sin (\hat{\varphi}%
_{j_{T}}t)\right\rangle .\;\;\;\;\;\;%
\end{array}
\end{eqnarray*}

From  (\ref{3.6}),  we have the following, for $k,j = 1,\ldots ,N,$
\begin{equation}
a_{kj}^{(1)}(T)=a_{kj}^{(2)}(T)=o(1),\quad  k\neq j,\quad a_{kk}^{(1)}(T)=\frac{1}{2}%
+o(1),  \label{3.26}
\end{equation}

\begin{eqnarray}
b_{kj}^{(1)}(T) &=& b_{kj}^{(2)}(T)=o(1),\quad k\neq j,\quad
b_{kk}^{(2)}(T)=\frac{1}{2}+o(1),  \label{3.27} \\
 \notag
\end{eqnarray}%
where  $o(1)$ means (different) stochastic processes  tending to
zero almost surely,   as $T\longrightarrow \infty .$

Then, one can continue as follows:
\begin{equation*}
c_{j}^{(1)}(T)=\left\langle \varepsilon (t),\cos (\widehat{\varphi }_{j_{T}}t)\right\rangle +\left\langle g(t,\theta
),\cos (\widehat{\varphi } _{j_{T}}t)\right\rangle= d_{j}^{(1)}(T)+d_{j}^{(2)}(T)
\end{equation*}%
\noindent where $d_{j}^{(1)}(T)=o_{P}(1)$ by Lemma \ref{lem22}, and
\begin{eqnarray*}
d_{j}^{(2)}(T) &=& A_{j}^{0}\left\langle \cos (\varphi _{j}t),\cos
(\hat{\varphi}_{j_{T}}t)\right\rangle +B_{j}^{0}\left\langle
\sin(\varphi_{j}t),\cos (\hat{\varphi}_{j_{T}}t)\right\rangle
+o_{P}(1),\nonumber\\
&=& \frac{1}{2}\left[ A_{j}z_{jT}-B_{j}y_{jT}\right]+o_{P}(1),
\end{eqnarray*}
\noindent or
\begin{equation}
c_{j}^{(1)}(T)=\frac{1}{2}\left[
A_{j}z_{j_{T}}-B_{j}y_{j_{T}}\right] +o_{P}(1),\ j=1,\ldots ,N,
\label{3.28}
\end{equation}

\noindent and similarly
\begin{equation}
c_{j}^{(2)}(T)=\frac{1}{2}\left[ A_{j}y_{j_{T}}+B_{j}z_{j_{T}}\right] +o_{P}(1),\ j=1,\ldots ,N, \label{3.29}
\end{equation}
\noindent where $o_{P}(1)$ are processes tending to zero in probability as $T\rightarrow \infty.$

Since $\left\vert z_{j_{T}}\right\vert ,\left\vert
y_{j_{T}}\right\vert \leq 1,$ we obtain, from (\ref{3.24}),
\begin{equation}
\left\vert \hat{A}_{k_{T}}\right\vert \left\vert
\hat{B}_{k_{T}}\right\vert \leq \left\vert A_{k}\right\vert
+\left\vert B_{k}\right\vert +o_{P}(1),\quad k=1,\ldots ,N.
\label{3.30}
\end{equation}

Let $\Delta g(t;\theta _{1},\theta _{2})=g(t,\theta _{1})-g(t,\theta
_{2}),$ and $K_{T}(\theta _{1},\theta _{2})=\left\langle \Delta
g(t;\theta _{1},\theta _{2}),\Delta g(t;\theta _{1},\theta
_{2})\right\rangle.$

By definition of LSE
\begin{equation}
Q_{T}(\hat{\theta}_{T})\leq Q_{T}(\theta ).  \label{3.31}
\end{equation}%
On the other hand,
\begin{equation}
Q_{T}(\hat{\theta}_{T})-Q_{T}(\theta _{0})=K_{T}(\hat{\theta}_{T},\theta
_{0})+2\left\langle \varepsilon (t),\Delta g(t;\theta _{1},\hat{\theta}%
_{T})\right\rangle ,  \label{3.32}
\end{equation}%
\noindent and by Lemma 2 and (\ref{3.28})-(\ref{3.30}), we have
\begin{equation}
\left\langle \varepsilon (t),\Delta g(t;\theta
_{1},\hat{\theta}_{T})\right\rangle =o_{P}(1).  \label{3.33}
\end{equation}

From (\ref{3.31}), (\ref{3.32}) and (\ref{3.33}), it follows  that
\begin{equation}
K_{T}(\hat{\theta}_{T},\theta )=o_{P}(1).  \label{3.34}
\end{equation}
\noindent
Consider $g_{k_{T}}(t)=\hat{A}_{k_{T}}\cos (\hat{\varphi}_{k_{T}}t)+\hat{B}_{k_{T}}\sin (%
\hat{\varphi}_{j_{T}}t)-A_{k}\cos (\varphi _{k}t)-B_{k}\sin (\varphi _{k}t).$ Observe that
\begin{equation}
K_{T}(\hat{\theta}_{T},\theta )=\sum_{k=1}^{N}\left\langle
g_{k_{T}}(t),g_{k_{T}}(t)\right\rangle +2\sum_{k<j}^{N}\left\langle
g_{k_{T}}(t),g_{j_{T}}(t)\right\rangle ,\quad k\neq j.
\label{3.31b}
\end{equation}
In a similar way as before, for $k =1,\ldots ,N,$
\begin{eqnarray*}
&&\left\langle g_{k_{T}}(t),g_{j_{T}}(t)\right\rangle =o_{P}(1),\quad k\neq j,\\
&&\left\langle g_{k_{T}}(t),g_{k_{T}}(t)\right\rangle =\frac{1}{2}\left[ \hat{A}_{k_{T}}^{2}+\hat{B}_{k_{T}}^{2}+(A_{k})^{2}+(B_{k})^{2}\right]\\
&&  \hspace*{2.5cm}-(\hat{A}_{k_{T}}A_{k}+\hat{B}%
_{k_{T}}B_{k})z_{k_{T}}+(\hat{A}_{k_{T}}B_{k}+\hat{B}%
_{k_{T}}A_{k})y_{k_{T}}+o_{P}(1).
\end{eqnarray*}

From (\ref{3.24})-(\ref{3.32}), we get
\begin{eqnarray}
K_{T}(\hat{\theta}_{T},\theta ) &=&\frac{1}{2}\sum_{k=1}^{N}\left( (A_{k})^{2}+(B_{k})^{2}\right) \left(
1-z_{k_{T}}^{2}-y_{k_{T}}^{2}\right)
+o_{P}(1)  \label{3.33b} \\
&=&\frac{1}{2}\sum_{k=1}^{N}\left( (A_{k})^{2}+(B_{k})^{2}\right)
\left( 1-\left( \frac{\sin \left(
\frac{1}{2}T(\hat{\varphi}_{k_{T}}-\varphi _{k})\right)
}{\frac{1}{2}T(\hat{\varphi}_{k_{T}}-\varphi _{k})}\right)
^{2}\right) +o_{P}(1).  \nonumber
\end{eqnarray}

Additionally, from (\ref{3.34}), we have
\begin{equation}
1-\left( \frac{\sin \left(
\frac{1}{2}T(\hat{\varphi}_{k_{T}}-\varphi _{k})\right)
}{\frac{1}{2}T(\hat{\varphi}_{k_{T}}-\varphi _{k})}\right)
=o_{P}(1),\quad  k=1,\ldots ,N.  \label{3.34b}
\end{equation}

Since $\frac{\sin x}{x},\ x\geq 0,$ is decreasing around zero, then, for $\epsilon \in (0,1),$ thanks to (\ref{3.34b}),

\begin{equation*}
P\left\{ \frac{1}{2}T\left|\hat{\varphi}_{k_{T}}-\varphi
\right|>\epsilon \right\} \leq P\left\{ \left( 1-\frac{\sin \left(
\frac{1}{2}T(\hat{\varphi}_{k_{T}}-\varphi )\right)
}{\frac{1}{2}T(\hat{\varphi}_{k_{T}}-\varphi )}\right) >1-\frac{\sin
\epsilon }{\epsilon }\right\} \longrightarrow 0,\quad
T\longrightarrow \infty ,
\end{equation*}%
\noindent or
\begin{equation}
T(\hat{\varphi}_{k_{T}}-\varphi )=o_{P}(1),\quad T\longrightarrow
\infty . \label{3.35}
\end{equation}

Observe that $z_{k_{T}}=\frac{\sin \left( \frac{1}{2}T(\hat{\varphi}_{k_{T}}-\varphi )\right) }{%
\frac{1}{2}T(\hat{\varphi}_{k_{T}}-\varphi )}\cos \left( \frac{1}{2}T(\hat{%
\varphi}_{k_{T}}-\varphi _{k})\right) .$ Since $1-\cos
x<\frac{x^{2}}{2},$ $x>0,$ using (\ref{3.35}), we see that, for any
$\epsilon >0,$
\begin{equation}
P\left\{ 1-\cos \left( \frac{1}{2}T(\hat{\varphi}_{k_{T}}-\varphi
)\right)
>\epsilon \right\} \leq P\left\{ T|\hat{\varphi}_{k_{T}}-\varphi
|>2\sqrt{2\epsilon }\right\} \longrightarrow 0,\quad T\longrightarrow \infty . \label{3.36}
\end{equation}
From (\ref{3.34b}) and (\ref{3.36}), we get that
$z_{k_{T}}=1+o_{P}(1),\ k=1,\ldots ,N.$ Moreover,
from (\ref{3.35}), we  obtain  $\sin \left( \frac{1}{2}T(\hat{\varphi}%
_{k_{T}}-\varphi )\right) =o_{P}(1),$ and thus,
$y_{k_{T}}=o_{P}(1),$ $k=1,\ldots ,N.$ Finally, from
(\ref{3.24}), we then have $$\hat{A}_{k_{T}}\overset{P}{\longrightarrow }A_{k},\ \hat{B}_{k_{T}}\overset{P}{%
\longrightarrow }B_{k},\quad k=1,\ldots ,N,\quad T\longrightarrow
\infty.$$

\end{proof}

\section{\textbf{Linearization and asymptotic uniqueness}}

This section reviews and clarifies a number of results, on non-linear
regression, in particular, from Ivanov and Leonenko (2004, 2008, 2009).
Consider the general non-linear regression model (\ref{nonregr1}) with the
noise process satisfying condition \textbf{A2}. Let $\hat{\theta}_{T}$ be
the LSE of an unknown parameter $\theta ,$ that is, a random vector $\hat{%
\theta}_{T}\in \Theta ^{c}$ having the property (\ref{3.4}). The following
assumption is considered.

\textbf{B1}. Suppose that $g(t,\tau )$ is twice differentiable with respect
to $\tau \in \Theta ^{c}$.

\noindent Under \textbf{B1}, we then get
\begin{equation}
g_{i}(t,\tau )=\frac{\partial }{\partial \tau _{i}}g(t,\tau ),\quad
g_{il}(t,\tau )=\frac{\partial ^{2}}{\partial \tau _{i}\partial \tau _{l}}%
g(t,\tau ),\quad i,l=1,\ldots ,q,  \label{4.1}
\end{equation}
\begin{equation*}
d_{T}(\tau )=diag\left( d_{iT}(\tau )\right)_{i=1}^{q},\ \tau \in \Theta
^{c}, \quad d_{iT}^{2}(\theta )=\int_{0}^{T}g_{i}^{2}(t,\theta )dt\quad
i=1,\ldots ,q.
\end{equation*}

\noindent Additionally, let us assume:

\noindent\ \ \ \ \textbf{B2} The following positive limits exist $\underline{%
\lim }_{T\rightarrow \infty }T^{-1}d_{iT}^{2}(\theta )>0,$ for $i=1,\ldots
,q.$

\noindent Note that the limits in (\ref{4.1}) can be, in particular, not
finite. Let also
\begin{equation}
d_{il,T}^{2}(\tau )=\int_{0}^{T}g_{il}^{2}(t,\tau )dt,\quad \tau \in \Theta
^{c},\quad \ i,l=1,\ldots ,q.  \label{4.2}
\end{equation}

Consider now the normalized LSE
\begin{equation}
\hat{u}_{T}=d_{T}(\theta )(\hat{\theta}_{T}-\theta ),  \label{4.3}
\end{equation}%
and the notation: $g(t,\theta +d_{T}^{-1}(\theta )u)=h(t,u),\ g_{i}(t,\theta
+d_{T}^{-1}(\theta )u)=h_{i}(t,u),i=1,\ldots ,q,\ g_{il}(t,\theta
+d_{T}^{-1}(\theta )u)=h_{il}(t,u),$ for $i,l=1,\ldots ,q,$ as well as $%
V(R)=\left\{ u\in \mathbb{R}^{q}:\left\Vert u\right\Vert <R\right\} $ for
the ball of radius $R.$ Here, the following change of variables is
performed: $u=d_{T}(\theta )(\tau -\theta ).$ The letter $k$ will be used
for denoting positive constants. The following assumptions are formulated,
for $R\geq 0,$ $\theta \in \Theta ,$ and $T>T_{0}(R)$ sufficiently large:

\noindent\ \ \ \ \ \textbf{B3} \quad $\sup_{t\in \left[ 0,T\right]
}\sup_{u\in V^{c}(R)}\frac{\left\vert h_{i}(t,u)\right\vert }{d_{iT}(\theta )%
}\leq k^{i}(R)T^{-1/2},$
\begin{equation}
\sup_{t\in \left[ 0,T\right] }\sup_{u\in V^{c}(R)}\frac{\left\vert
h_{il}(t,u)\right\vert }{d_{il,T}(\theta )}\leq k^{il}(R)T^{-1/2},\
\sup_{t\in \left[ 0,T\right] }\sup_{\theta \in V^{c}(R)}\frac{%
d_{il,T}(\theta )}{d_{iT}(\theta )d_{lT}(\theta )}\leq \tilde{k}%
^{il}(R)T^{-1/2},\quad i,l=1,\ldots ,q.  \label{4.4}
\end{equation}

We will use the notation: $H(t;u_{1},u_{2})=h(t,u_{1})-h(t,u_{2}),$ $%
H_{i}(t;u_{1},u_{2})=h_{i}(t,u_{1})-h_{i}(t,u_{2}),$ for $i=1,\ldots,q.$
Introduce also the vectors $\psi _{T}(u)=\left( \psi _{T}^{i}(u)\right)
_{i=1}^{q},$ with
\begin{equation}
\psi _{T}^{i}(u)=\int_{0}^{T}\varepsilon (t)\frac{h_{i}(t,u)}{d_{iT}(\theta )%
}dt+\int_{0}^{T}H(t;0,u)\frac{h_{i}(t,u)}{d_{iT}(\theta )}dt,  \label{4.7}
\end{equation}%
and $L_{T}(u)=\left( L_{T}^{i}(u)\right) _{i=1}^{q}$, with
\begin{equation}
L_{T}^{i}(u)=\int_{0}^{T}\left( \varepsilon (t)-\sum_{l=1}^{q}\frac{%
g_{l}(t,\theta )}{d_{lT}(\theta )}u_{l}\right) \frac{g_{i}(t,\theta )}{%
d_{iT}(\theta )}dt,\ \ i=1,\ldots ,q.  \label{4.8}
\end{equation}

The vectors (\ref{4.7}) and (\ref{4.8}) are defined for $u\in
U_{T}^{c}(\theta ),$ where $U_{T}(\theta )=d_{T}(\theta )(\Theta -\theta ).$
Note that, under our assumptions, for any $R>0,$ $V^{c}(R)\subset
U_{T}(\theta ),$ for $T>T_{0}(R).$

The normalized LSE $\hat{u}_{T}$ satisfies the system of normal equations:%
\begin{equation}
\psi _{T}(u)=0,  \label{4.9}
\end{equation}%
while the vector $L_{T}(\theta )$ corresponds to the auxiliary linear
regression model:%
\begin{equation}
Z(t)=\sum_{i=1}^{q}g_{i}(t,\theta )\beta _{i}+\varepsilon (t),\ t\in \left[
0,T\right] .  \label{4.10}
\end{equation}

The system of normal equations for the linear regression model (\ref{4.10})
\begin{equation}
L_{T}(\theta )=0,  \label{4.11}
\end{equation}%
\noindent determines the normed LSE $\tilde{u}_{T}$ of the parameter $\beta
, $ if
\begin{equation}
\tilde{u}_{T}=d_{T}(\theta )(\tilde{\beta}_{T}-\beta ),  \label{4.12}
\end{equation}%
where $\tilde{\beta}_{T}$ is the ordinary LSE of the parameter $\beta $ in
the model (\ref{4.10}).

\begin{theorem}
\label{th228} Under the assumptions \textbf{A1}-\textbf{A3} and \textbf{B1}-%
\textbf{B3}, for any $R>0,$ $r>0$
\begin{equation}
P\left\{ \sup_{u\in V^{c}(R)}\left\Vert \psi _{T}(u)-L_{T}(u)\right\Vert
>r\right\} \longrightarrow 0,\ T\longrightarrow \infty .  \label{4.13}
\end{equation}
\end{theorem}

The proof of Theorem \ref{th228} is given in Appendix 1.

In this section, we show that the LSE $\hat{\theta}_{T}$ in certain sense is
the asymptotically unique solution of the system of normal equations (\ref%
{4.9}) as $T\longrightarrow \infty .$ Let us first consider $J_{T}(\theta
)=\left( J_{il,T}(\theta )\right) _{i,l=1}^{q},$ where
\begin{equation}
J_{il,T}(\theta )=d_{iT}^{-1}(\theta )d_{lT}^{-1}(\theta
)\int_{0}^{T}g_{i}(t,\theta )g_{l}(t,\theta )dt.  \label{5.1}
\end{equation}
Denote by $\lambda _{\min }(A)$ and $\lambda _{\max }(A)$ the respective
minimal and maximal eigenvalues of a positive definite matrix $A.$ Let us
formulate the next condition:

\noindent\ \ \ \ \textbf{B4} For some $\lambda _{\ast }>0$ and $T>T_{0},$
\begin{equation*}
\lambda _{\min }\left( J_{T}(\theta )\right) \geq \lambda _{\ast }.
\end{equation*}

Consider now the normed LSE
\begin{equation}
\hat{w}_{T}=T^{-1/2}d_{T}(\theta )(\hat{\theta}_{T}-\theta ),  \label{5.2}
\end{equation}%
\noindent where the change of variables $w=T^{-1/2}d_{T}(\theta )(\tau
-\theta )$ is applied into the regression function and its derivatives. The
following notation is established: For $i,l=1,\ldots q,$
\begin{equation*}
f(t,w)=g(t,\theta +T^{1/2}d_{T}^{-1}(\theta )w),\ f_{i}(t,w)=g_{i}(t,\theta
+T^{1/2}d_{T}^{-1}(\theta )w),\ f_{il}(t,w)=g_{il}(t,\theta
+T^{1/2}d_{T}^{-1}(\theta )w).
\end{equation*}

Additionally, we denote, for $i,l=1,\ldots q,$ $%
F(t;w_{1},w_{2})=f(t,w_{1})-f(t,w_{2}),$
\begin{equation*}
F_{i}(t;w_{1},w_{2}) = f_{i}(t,w_{1})-f_{i}(t,w_{2}),\ F_{il}(t;w_{1},w_{2})
= f_{il}(t,w_{1})-f_{il}(t,w_{2}), \ \Phi
_{il,T}(w,0)=\int_{0}^{T}F_{il,T}^{2}(t;w,0)dt.
\end{equation*}

Finally, the following assumption is considered in the derivation of Theorem %
\ref{uniquiness} below:%
\begin{equation*}
\end{equation*}

\noindent\ \ \ \ \textbf{B5} For some $\tau _{0}>0,$ and for $i,l=1,\ldots
q, $ \quad $\sup_{t\in \left[ 0,T\right] }\sup_{w\in V^{c}(\tau _{0})}\frac{%
\left\vert f_{i}(t,w)\right\vert }{d_{iT}(\theta )}\leq \tilde{k}^{i}(\tau
_{0})T^{-1/2},$
\begin{equation*}
\sup_{t\in \left[ 0,T\right] }\sup_{w\in V^{c}(\tau _{0})}\frac{\left\vert
f_{il}(t,w)\right\vert }{d_{il,T}(\theta )}\leq \tilde{k}^{il}(\tau
_{0})T^{-1/2},\ \sup_{w\in V^{c}(\tau _{0})}Td_{iT}^{-2}(\theta
)d_{lT}^{-2}(\theta )\Phi _{il,T}(w,0)\left\Vert w\right\Vert ^{-2}\leq \hat{%
k}_{il}(\tau _{0}).
\end{equation*}

Consider the functional
\begin{equation*}
(2T)^{-1}\int_{0}^{T}\left[ x(t)-f(t,w)\right] ^{2}dt=\frac{1}{2}%
Q_{T}(\theta +T^{1/2}d_{T}^{-1}(\theta )w),
\end{equation*}%
\noindent and the vector
\begin{equation*}
\mathbf{M}_{T}(w)=\left( M_{T}^{i}(w)\right) _{i=1}^{q}=\left( \frac{%
\partial }{\partial w_{i}}\left( \frac{1}{2}Q_{T}(\theta
+T^{1/2}d_{T}^{-1}(\theta )w)\right) \right) _{i=1}^{q} =\left(
T^{-1/2}\int_{0}^{T}\left[ x(t)-f(t,w)\right] \frac{-f_{i}(t,w)}{%
d_{iT}(\theta ) }dt\right) _{i=1}^{q}.
\end{equation*}

Then, the normed LSE (\ref{5.2}) satisfies the system of equations
\begin{equation}
\mathbf{M}_{T}(w)=0.  \label{5.3}
\end{equation}%
\noindent\ \ \ \ \noindent\ \ \ \ \noindent\ \ \ \

\noindent \textbf{C} For any $r>0,$
\begin{equation*}
P\left\{ \left\Vert \hat{w}_{T}\right\Vert >r\right\} \longrightarrow
0,\quad T\longrightarrow \infty .
\end{equation*}

Note that if the normed LSE $\hat{w}_{T}$ is an unique solution of the
system of equations (\ref{5.3}), then, the LSE $\hat{u}_{T}$ is the unique
solution of the system (\ref{4.9}).

\begin{theorem}
\label{uniquiness} Under conditions \textbf{A1}-\textbf{A3}, \textbf{B1}-%
\textbf{B5} and \textbf{C}, the normed LSE (\ref{5.2}) is an unique solution
of the system of equations (\ref{5.3}) with probability tending to $1,$ as $%
T\longrightarrow \infty .$
\end{theorem}

We place the proof of this theorem into Appendix 2.

\begin{remark}
The verification of the conditions \textbf{B3} - \textbf{B5} fulfilment for
regression function (\ref{nonregr2}) is not difficult and we omit it.
\end{remark}

\section{\textbf{Central limit theorem}}

This section is derivation of the convergence to the Gaussian distribution.
This convergence result is obtained, under conditions, for the integral
functional

\begin{equation}
\zeta _{T}=d_{T}^{-1}(\theta )\int\limits_{0}^{T}\nabla g(t,\theta )G(\xi
(t))dt,  \label{6.1}
\end{equation}%
\noindent as $T\longrightarrow \infty ,$ where $g(t,\theta )$ is the general
regression function and $\nabla g(t,\theta )=\left( g_{i}(t,\theta )\right)
_{i=1}^{q}$ is its gradient. We introduce a family of a matrix-valued
measures $\boldsymbol{\mu }_{T}(d\lambda )=\left( \mu _{T}^{jl}(d\lambda
,\theta )\right) _{j,l=1}^{q},T>0,$ where $\mu _{T}^{jl}(\lambda ,\theta )$
are given by (\ref{regmes1}). Using the notation (\ref{regmes1})-(\ref%
{regmes2}) and (\ref{4.1}) note that
\begin{equation*}
d_{jT}^{2}(\theta )=\frac{1}{2\pi }\int_{\mathbb{R}}\left\vert
g_{T}^{j}(\lambda ,\theta )\right\vert ^{2}d\lambda ,\quad j=1,\dots ,q.
\end{equation*}

\noindent \textbf{B6} The family of measures $\boldsymbol{\mu }_{T}(d\lambda
)$ converges weakly to the measure $\boldsymbol{\mu }(d\lambda ,\theta
)=\left( \mu ^{jl}(d\lambda ,\theta )\right) _{j,l=1}^{q},$ as $%
T\longrightarrow \infty .$

Condition \textbf{B6} means that the elements $\mu ^{jl}(d\lambda ,\theta )$
of the matrix $\boldsymbol{\mu }(d\lambda ,\theta )$ are complex signed
measures of bounded variation and the matrix $\boldsymbol{\mu }(\mathcal{A}%
,\theta )$ is positive semi-definite for any Borel set $\mathcal{A}.$ The
limiting measure $\boldsymbol{\mu }(d\lambda ,\theta )$ is called the
spectral measure of the regression function $g(t,\theta )$, see Grenander
and Rosenblatt (1984), Holevo (1976), Ibragimov and Rozanov (1980), Ivanov
and Leonenko (1989). Note that
\begin{eqnarray*}
\int_{\mathbb{R}}\boldsymbol{\mu }_{T}(d\lambda ,\theta ) &=&\left( \int_{%
\mathbb{R}}\mu _{T}^{jl}(d\lambda ,\theta )\right) _{j,l=1}^{q} \\
&=&\left( d_{jT}^{-1}(\theta )d_{lT}^{-1}(\theta
)\int\limits_{0}^{T}g_{j}(t,\theta )g_{l}(t,\theta )dt\right)
_{j,l=1}^{q}=\left( J_{jl,T}(\theta )\right) _{j,l=1}^{q}=J_{T}(\theta ),
\end{eqnarray*}%
and $J_{T}(\theta )\longrightarrow J(\theta ),$ as $T\longrightarrow \infty
, $ for matrix $J(\theta )>0,$ such that $\lambda _{\min }\left( J(\theta
)\right) \geq \tilde{\lambda}_{\ast }>0$. We then have
\begin{equation*}
J(\theta )=\left( \int_{\mathbb{R}}\mu ^{jl}(d\lambda ,\theta )\right)
_{j,l=1}^{q}=\int_{\mathbb{R}}\boldsymbol{\mu }(d\lambda ,\theta ).
\end{equation*}

Note also that
\begin{equation*}
\zeta _{T}=\left( \int\limits_{0}^{T}G(\xi (t))v_{iT}(t,\theta )dt\right)
_{i=1}^{q},
\end{equation*}%
\noindent where
\begin{equation*}
v_{iT}(t,\theta )=g_{i}(t,\theta )d_{iT}^{-1}(t,\theta ).
\end{equation*}%
The weak convergence of the random vector
\begin{equation}
\zeta _{T}=\left( \zeta _{1T},\ldots ,\zeta _{qT}\right) \Rightarrow \left(
\zeta _{1},\ldots ,\zeta _{q}\right) =\zeta ,  \label{6.2}
\end{equation}%
\noindent is equivalent to convergence of the characteristic functions: for
any $z\in \mathbb{R}^{q}$,%
\begin{equation}
\mathit{E}\exp \left\{ iu\left\langle \zeta _{T},z\right\rangle \right\}
\underset{T\longrightarrow \infty }{\longrightarrow }\mathit{E}\exp \left\{
iu\left\langle \zeta ,z\right\rangle \right\} ,\ u\in \mathbb{R}^{1}.
\label{6.3}
\end{equation}%
\noindent Thus, the convergence in (\ref{6.2}) will follow from (\ref{6.3}).

Under the condition \textbf{A3}, consider expansion (\ref{3.1})

\begin{equation*}
G(x)=\sum\limits_{k=1}^{\infty }\frac{C_{k}}{k!}H_{k}(x).  \label{6.4}
\end{equation*}

For $z=\left( z_{1},\ldots ,z_{q}\right) \in \mathbb{R}^{q}$, we denote
\begin{equation*}
\sum\limits_{i=1}^{q}z_{i}v_{iT}(t,\theta )=R_{T}(t,\theta ,z)=R_{T}(t).
\end{equation*}%
\noindent Then,
\begin{equation*}
\left\langle \zeta _{T},z\right\rangle =\int\limits_{0}^{T}G(\xi
(t))R_{T}(t)dt=\sum\limits_{k=m}^{\infty }\frac{C_{k}}{k!}%
\int\limits_{0}^{T}H_{m}(\xi (t))R_{T}(t)dt,
\end{equation*}%
\noindent in the Hilbert space $\mathit{L}_{2}\left( \Omega ,\mathbb{F}%
,P\right) ,$ and
\begin{equation}
\mathit{E}\left\langle \zeta _{T},z\right\rangle
^{2}=\sum\limits_{k=m}^{\infty }\frac{C_{k}^{2}}{k!}\int\limits_{0}^{T}\int%
\limits_{0}^{T}B^{k}(t-s)R_{T}(t)R_{T}(s)dtds.  \label{6.5}
\end{equation}

The following condition is now assumed:

\noindent\ \ \ \ \ \ \textbf{A4} Either 1) Hrank$(G)=1,$ $\alpha >1;$ or 2)
Hrank$(G)=m\geq 2,$ $\alpha m>1;$ where $\alpha =\min_{j=0,1,\ldots
,\kappa}\alpha _{j}.$

In the further reasoning we use the part 2) of condition \ \textbf{A4.}

For $k\geq m\geq 2,$ let
\begin{equation}
f^{(\ast k)}(\lambda )=\int_{R^{k-1}}f(\lambda -\lambda _{2}-\cdots -\lambda
_{k})\prod\limits_{i=2}^{k}f(\lambda _{i})d\lambda _{2}\cdots d\lambda _{k},
\label{6.6}
\end{equation}%
\noindent the $k$-th convolution of the spectral density given under
assumption {\bfseries A1}. Under the condition \textbf{A4}, $B^{k}(t)\in
\mathit{L}_{1}(\mathbb{R}),$ $k\geq m.$ Thus, all convolutions $f^{(\ast
k)}(\lambda ),$ $k\geq m,$ are continuous and bounded functions under
\textbf{A4}, and
\begin{eqnarray}
\sigma _{T}^{2}(k,z)
&=&\int\limits_{0}^{T}\int\limits_{0}^{T}B^{k}(t-s)R_{T}(t)R_{T}(s)dtds=\sum%
\limits_{i,j=1}^{q}\left(
\int\limits_{0}^{T}\int\limits_{0}^{T}B^{k}(t-s)v_{iT}(t,\theta
)v_{jT}(s,\theta )dtds\right) z_{i}z_{j}  \notag \\
&=&2\pi \sum\limits_{i,j=1}^{q}\int_{\mathbb{R}}f^{(\ast k)}(\lambda )\mu
_{T}^{i,j}(d\lambda ,\theta )z_{i}z_{j}\longrightarrow 2\pi \int_{\mathbb{R}%
}f^{(\ast k)}(\lambda )m_{z}(d\lambda ,\theta )=\sigma _{k}^{2}(z),
\label{6.7}
\end{eqnarray}%
where $m_{z}(d\lambda ,\theta )=\sum\limits_{i,j=1}^{q}\mu ^{i,j}(d\lambda
,\theta )z_{i}z_{j}$ is a measure. Thus, as $T\longrightarrow \infty ,$
\begin{equation}
\mathit{E}\left\langle \zeta _{T},z\right\rangle ^{2}\longrightarrow 2\pi
\sum\limits_{k=m}^{\infty }\frac{C_{k}^{2}}{k!}\int_{\mathbb{R}}f^{(\ast
k)}(\lambda )m_{z}(d\lambda ,\theta )=\sum\limits_{k=m}^{\infty }\frac{%
C_{k}^{2}}{k!}\sigma _{k}^{2}(z)=\sigma ^{2}(z),  \label{6.7b}
\end{equation}

To prove asymptotic normality, the method of moments can be applied. That
is, for any integer $p\geq 2,$ it will be showed that
\begin{equation}
\lim_{T\rightarrow \infty }\mathit{E}\eta _{T}^{p}=\mathit{E}\eta
^{p}=\left\{
\begin{array}{l}
(p-1)!!\sigma ^{p}(z),\quad p=2\nu ,\quad \nu =1,2,\ldots , \\
0,\quad p=2\nu +1,\quad \nu =1,2,\ldots , \\
\end{array}%
\right.  \label{6.8}
\end{equation}

\noindent where
\begin{equation*}
\eta _{T}=\sum_{k=m}^{\infty }\frac{C_{k}}{k!}\int_{0}^{T}H_{k}(\xi
(t))R_{T}(t)dt,
\end{equation*}%
\noindent and $\eta \sim N(0,\sigma ^{2}(z)).$ Let
\begin{equation*}
\eta _{T}=\eta _{T}(M)+\eta _{T}^{\prime }(M),\quad \eta
_{T}(M)=\sum_{k=m}^{M}\frac{C_{k}}{k!}\int_{0}^{T}H_{k}(\xi (t))R_{T}(t)dt.
\end{equation*}

\begin{lemma}
\label{lem4} Assume that the conditions \textbf{A1}-\textbf{A3} and \textbf{%
B1}-\textbf{B3} are satisfied and, for any $M\geq m,$
\begin{equation}
\eta _{T}(M)\underset{T\rightarrow \infty }{\Longrightarrow }\eta (M)\sim
N(0,\sigma _{M}^{2}(z)),  \label{6.9}
\end{equation}%
\noindent where%
\begin{equation}
\sigma _{M}^{2}(z)=\sum_{k=m}^{M}\frac{C_{k}^{2}}{k!}\sigma _{k}^{2}(z),
\label{6.10}
\end{equation}%
\noindent then,
\begin{equation}
\eta _{T}\Longrightarrow \eta \sim N(0,\sigma ^{2}(z)).  \label{6.11}
\end{equation}
\end{lemma}

\begin{proof}
Note that, uniformly in $T,$
\begin{equation}
E\left[ \eta _{T}^{\prime }(M)\right] ^{2}\longrightarrow 0,\quad M\longrightarrow \infty . \label{6.12}
\end{equation}%
\noindent Specifically, under {\bfseries B3}, with $u=0,$ for $k(0)
=(k_{1}(0),\ldots ,k_{q}(0)),$
\begin{eqnarray}
\left| R_{T}(t)\right| &\leq &\left|
\sum_{i=1}^{q}z_{i}v_{iT}(t,\theta )\right| \leq
T^{-1/2}\sum_{i=1}^{q}\left| z_{i}\right| k_{i}(0)\leq
T^{-1/2}\left\| z\right\| \cdot \left\| k(0)\right\| .\;
\label{6.13}
\end{eqnarray}
Therefore,
\begin{equation*}
E\left[ \eta _{T}^{\prime }(M)\right] ^{2}=
\end{equation*}%
\begin{equation*}
=E\left( \sum_{k=M+1}^{\infty }\frac{C_{k}}{k!}\int_{0}^{T}H_{k}(\xi
(t))R_{T}dt\right) ^{2}=\sum_{k=M+1}^{\infty }\frac{C_{k}^{2}}{k!}%
\int_{0}^{T}\int_{0}^{T}B^{k}(t-s)R_{T}(t)R_{T}(s)dtds
\end{equation*}
\begin{eqnarray*}
&\leq &\left\| z\right\| ^{2}\left\|k(0)\right\|^{2}\frac{1}{T}\int_{0}^{T}%
\int_{0}^{T}|B^{m}(t-s)|dtds\sum_{k=M+1}^{\infty }\frac{C_{k}^{2}}{k!} \\
&\leq &\left\| z\right\| ^{2}\left\|k(0)\right\|^{2}\int_{-\infty }^{\infty }|B^{m}(t)|dt\sum_{k=M+1}^{\infty
}\frac{C_{k}^{2}}{k!}=\pi (M)\longrightarrow 0,\;M\longrightarrow \infty ,
\end{eqnarray*}
\noindent since $\sum_{k=m}^{\infty }\frac{C_{k}^{2}}{k!}=EG^{2}(\xi
(0))<\infty .$ Thus, for any $\epsilon >0,$ uniformly in $T,$
$$P\left\{ \left| \eta _{T}^{\prime }(M)\right| >\epsilon \right\}
\leq \frac{\pi (M)}{\epsilon ^{2}}\longrightarrow 0,\quad
M\longrightarrow \infty .$$ We then obtain
\begin{eqnarray*}
P\left\{ \eta _{T}\leq x\right\} &=& P\left\{ \eta _{T}\leq x, |\eta
_{T}^{\prime }(M)|\leq \epsilon \right\} +P\left\{
\eta _{T}\leq x, |\eta_{T}^{\prime }(M)|>\epsilon \right\} \\
&\leq & P\left\{ \eta _{T}(M)+\eta _{T}^{\prime }(M)\leq x, |\eta
_{T}^{\prime }(M)|\leq \epsilon \right\}
+\frac{\pi (M)}{\epsilon^{2}} \\
&\leq &P\left\{ \eta _{T}(M)\leq x+\epsilon \right\} +\frac{\pi (M)}{%
\epsilon ^{2}}.
\end{eqnarray*}
Thus, we have the following
\begin{equation}
\overline{\lim_{T\rightarrow \infty }}P\left\{ \eta _{T}\leq
x\right\} \leq \Phi _{M}(x+\epsilon )+\frac{\pi (M)}{\epsilon ^{2}},
\label{6.14}
\end{equation}%
where $\Phi _{M}$ is the d.f. of $\eta (M)\sim N(0,\sigma_{M} ^{2}(z)).$

On the other hand,
\begin{eqnarray*}
P\left\{ \eta _{T}(M)\leq x-\epsilon \right\} &=& P\left\{ \eta
_{T}(M)+\epsilon \leq x, |\eta _{T}^{\prime }(M)|\leq \epsilon
\right\}+P\left\{ \eta _{T}(M)+\epsilon \leq x,\left| \eta
_{T}^{\prime
}(M)\right| >\epsilon \right\}\nonumber \\
&\leq & P\left\{ \eta _{T}\leq x\right\} +\frac{\pi (M)}{\epsilon
^{2}},
\end{eqnarray*}

\noindent or equivalently,
$$
P\left\{ \eta _{T}\leq x\right\} \geq P\left\{ \eta _{T}(M)\leq x-\epsilon\right\} - \frac{\pi (M)}{\epsilon ^{2}},$$
\noindent which leads to

\begin{equation}
\Phi _{M}(x-\epsilon )- \frac{\pi (M)}{\epsilon ^{2}}\leq \underline{\lim}_{T\longrightarrow \infty}P\left\{ \eta
_{T}\leq x\right\}.\label{6.15b}
\end{equation}

\noindent  Taking  the limit in (\ref{6.14}) and (\ref{6.15b}) in
$M,$ we obtain that, for any $\epsilon >0,$
\begin{equation*}
\Phi _{\infty}(x-\epsilon )\leq \underline{\lim}_{T\longrightarrow
\infty}P\left\{ \eta _{T}\leq x\right\} \leq
\overline{\lim_{T\longrightarrow \infty }}P\left\{ \eta _{T}\leq
x\right\} \leq \Phi_{\infty }(x+\epsilon ),
\end{equation*}%
where $\Phi _{\infty }$ is the d.f. of a normal random variable with zero mean and variance $\sigma ^{2}(z)$. As
$\epsilon \longrightarrow 0,$ we then have the desired result.
\end{proof}

We therefore need to prove that, for any integer $p\geq 2,$ and for fixed $%
M\geq m,$
\begin{equation}
\lim_{T\rightarrow \infty }\mathit{E}\eta _{T}^{p}(M)=\mathit{E}\eta
^{p}(M)=\left\{
\begin{array}{l}
(p-1)!!\sigma _{M}^{p}(z),\quad p=2\nu ,\quad \nu =1,2,\ldots , \\
0,\quad p=2\nu +1,\quad \nu =1,2,\ldots , \\
\end{array}%
\right.  \label{6.16}
\end{equation}

\noindent where $\eta (M)\sim N(0,\sigma _{M}^{2}(z)).$ The following
notation is considered:
\begin{equation*}
D_{p}=\left\{ J:J=\left( l_{1},\dots,l_{p}\right) ,\quad 1\leq l_{i}\leq
M,\quad i=1,\dots,p\right\} ,
\end{equation*}%
\begin{equation*}
K\left( J\right) =\prod_{i=1}^{p}\frac{C_{j_{i}}}{\left( j_{i}\right) !},
\quad \int^{\left( p\right) }\dots=\prod_{i=1}^{p}\int_{0}^{T}\dots
\end{equation*}

Note that by diagram formula (see Lemma \ref{lem2})%
\begin{equation*}
\mathit{E}\zeta _{T}^{p}(M)=E\left( \sum_{k=m}^{M}\frac{C_{k}}{k!}%
\int_{0}^{T}H_{k}\left( \xi \left( t\right) \right) R_{T}(t)dt\right) ^{p}=
\end{equation*}%
\begin{equation}
=\sum_{D_{p}}K\left( J\right) \int^{\left( p\right)
}\prod_{j=1}^{p}R_{T}(t_{j})\sum_{\Gamma \in \mathit{L}(J\mathrm{)}%
}\prod_{\varpi \in R\left( \Gamma \right) }B\left( t_{d_{1}\left( \varpi
\right) }-t_{d_{2}\left( \varpi \right) }\right) dt_{1}\dots dt_{p}.
\label{6.17}
\end{equation}

Let $\mathit{L}^{\ast }(J)$ be a set of regular diagrams. We split the sum
\begin{equation*}
\sum_{\Gamma \in \mathit{L}}...=\sum_{\Gamma \in \mathit{L}^{\ast
}}...+\sum_{\Gamma \in \mathit{L}\backslash \mathit{L}^{\ast }}...
\end{equation*}%
and denote
\begin{equation*}
\sum_{\Gamma \in \mathit{L}^{\ast }}...\equiv \sum_{p}^{\ast }\left(
T\right),\quad \sum_{\Gamma \in \mathit{L}\backslash \mathit{L}^{\ast
}}...\equiv \sum_{p}\left( T\right) .
\end{equation*}
\noindent We will study their behavior separately.

\medskip

\noindent \textit{Analysis of the regular diagrams:} \medskip

If $p=2\nu +1$ is odd, then $L^{\ast }=\emptyset ,$ $J\in D_{p},$ and $%
\lim_{T\rightarrow \infty }\sum_{p}^{\ast }\left( T\right) =0.$ If $p=2\nu ,$
for an arbitrary fix regular diagram $\Gamma \in L^{\ast }(J),$ $J\in D_{p},$
which has $2m_{j}$ levels of cardinality $r_{j},$ with $m\leq r(j)\leq M,$ $%
j=1,\dots ,l,$ $\sum\limits_{j=1}^{l}m_{j}=\nu ,$ where $1\leq l\leq \nu $
is fixed, and all $r(j),$ $j=1,\dots ,l,$ are different, we obtain that the
contribution to $\sum_{p}^{\ast }\left( T\right) $ is equal to
\begin{equation}
\prod_{j=1}^{l}\left( \frac{C_{r(j)}}{\left( r(j)\right) !}\right)
^{2m_{j}}\sigma _{T}^{2m_{j}}(r(j),z)\underset{T\longrightarrow \infty }{%
\longrightarrow }\prod_{j=1}^{l}\left( \frac{C_{r(j)}}{\left( r(j)\right) !}%
\right) ^{2m_{j}}\sigma ^{2m_{j}}(r(j),z).  \label{6.18}
\end{equation}

Note that the number of regular diagrams with $2m_{j}$ levels of cardinality
$r\left( j\right) $ , $j=1,\dots ,l,$ $\sum\limits_{j=1}^{l}m_{j}=\nu $ is
equal to
\begin{equation*}
\frac{(2\nu )!}{(2m_{1})!\cdot \cdot \cdot (2m_{l})!}\left\{
\prod_{j=1}^{l}(2m_{j}-1)(2m_{j}-3)\cdot \cdot \cdot 3\cdot 1\right\}
\left\{ \prod_{j=1}^{l}(r\left( j\right) !)^{m_{j}}\right\} =
\end{equation*}%
\begin{equation}
=\frac{\left( 2\nu -1\right) !!\left( \nu \right) !}{m_{1}!\dots m_{l}!}%
\prod_{j=1}^{l}(r\left( j\right) !)^{m_{j}}.  \label{6.19}
\end{equation}

From (\ref{6.17})-(\ref{6.19}), we obtain
\begin{equation*}
\sum_{p}^{\ast }\left( T\right) =\sum_{J\in D_{p}}K\left( J\right)
\int^{\left( p\right) }\prod_{j=1}^{p}R_{T}(t_{j})\sum_{\Gamma \in
L^{*}(J)}\prod_{\varpi \in R \left( \Gamma\right) }B\left( t_{d_{1}\left(
\varpi \right) }-t_{d_{2}\left( \varpi \right) }\right) dt_{1}\dots dt_{p}
\end{equation*}

\begin{eqnarray}
&=&\left( 2\nu -1\right) !!\sum_{1<l<\nu }\sum_{m_{1}+...+m_{l}=\nu
}\sum\limits_{\substack{ m\leq r(j)\leq M,  \\ 1\leq j\leq l}}\frac{\nu !}{%
m_{1}!...m_{l}!}\prod_{j=1}^{l}\left[ \frac{C_{r\left( j\right) }^{2}}{%
r\left( j\right) !}\sigma _{T}^{2}\left( r\left( j\right) ,z\right) \right]
^{m_{j}}=  \notag \\
&=&\left( p-1\right) !!\left[ \sum_{r\left( j\right) =m}^{M}\frac{C_{r\left(
j\right) }^{2}}{r\left( j\right) !}\sigma _{T}^{2}\left( r\left( j\right)
,z\right) \right] ^{\frac{p}{2}}\underset{T\longrightarrow \infty }{%
\longrightarrow }\left( p-1\right) !!\left[ \sum_{r\left( j\right) =m}^{M}%
\frac{C_{r\left( j\right) }^{2}}{r\left( j\right) !}\sigma ^{2}\left(
r\left( j\right) ,z\right) \right] ^{\frac{p}{2}}=  \notag
\end{eqnarray}%
\begin{equation}
=(p-1)!!\sigma _{M}^{p}(z).  \label{6.20}
\end{equation}%
\medskip

\noindent \textit{Analysis of the nonregular diagrams:}

\medskip

We now wish to prove that

\begin{equation}
\sum_{p}\left( T\right) =\sum_{J\in D_{p}}K\left( J\right) \sum_{\Gamma \in
\mathit{L}\backslash \mathit{L}^{\ast }}I_{\Gamma }(J,T)\longrightarrow 0,
\label{6.21}
\end{equation}%
\noindent where
\begin{equation*}
I_{\Gamma }(J,T)=\int^{\left( p\right)
}\prod_{j=1}^{p}R_{T}(t_{j})\prod_{\varpi \in R \left( \Gamma\right)
}B\left( t_{d_{1}\left( \varpi \right) }-t_{d_{2}\left( \varpi \right)
}\right) dt_{1}\dots dt_{p}.
\end{equation*}

Now, we assume that the diagram $\Gamma ^{\prime }\in \mathit{L}%
(l_{1}^{\prime },\dots ,l_{p}^{\prime })$ satisfies $l_{1}^{\prime }\leq
\dots \leq l_{p}^{\prime }.$ We then have
\begin{equation}
\left\vert I_{\Gamma }(J,T)\right\vert \leq \left\Vert z\right\Vert
^{p}\left\Vert k(0)\right\Vert ^{p}T^{-p/2}\int^{\left( p\right) }\prod
_{\substack{ \varpi \in R(\Gamma )  \\ d_{1}\left( \varpi \right) =j}}%
B\left( t_{j}-t_{d_{2}\left( \varpi \right) }\right) dt_{1}\dots dt_{p}.
\label{6.24}
\end{equation}%
Given a permutation $\pi $ of the set $\left( 1,\dots ,p\right) $ and the
diagram $\Gamma \in \mathit{L}(l_{1},\dots ,l_{p}),$ we define the diagram $%
\pi \Gamma $ in the following way: the $\pi (j)$ level of $\pi \Gamma $ has
cardinality $l_{j},j=1,\dots ,p,$ and $\varpi =((j_{1},k_{1}),\
(j_{2},k_{2}))\in R(\Gamma )$ if and only if $\pi (\varpi )=((\pi
(j_{1}),k_{1}),\ (\pi (j_{2}),k_{2}))\in \pi R(\Gamma ).$ Given a diagram $%
\Gamma \in \mathit{L}(l_{1},\dots ,l_{p}),$ we define the integer-valued
function $q_{\Gamma }(j)$ on the set $\{1,\dots ,p\}$ in the following way: $%
q_{\Gamma }(j)$ is the cardinality of the edges $\varpi \in R(\Gamma )$ such
that $d_{1}(\varpi )=j.$

Observe that for $\Gamma \in L(l_{1},\dots ,l_{p}),$ and $J=(l_{1},\dots
,l_{p})$
\begin{equation}
I_{\Gamma }\left( J,T\right) =I_{\pi \Gamma }\left( \pi J,T\right) ,\quad
\pi J=(l_{\pi ^{-1}(1)},\dots ,l_{\pi ^{-1}(p)}).  \label{6.22}
\end{equation}%
For all the diagrams $\Gamma ,$ there exists a permutation $\pi $ such that $%
\Gamma ^{\prime }=\pi \Gamma $ has the following property: $\Gamma ^{\prime
}\in \mathit{L}(l_{1}^{\prime },\dots ,l_{p}^{\prime })$ and
\begin{equation}
l_{1}^{\prime }\leq \dots \leq l_{p}^{\prime }.  \label{6.23}
\end{equation}

Then, for $q_{\Gamma }(j)\geq 1,$
\begin{equation*}
\int_{0}^{T}\prod_{\substack{ \varpi \in R(\Gamma )  \\ d_{1}\left( \varpi
\right) =j}}B\left( t_{j}-t_{d_{2}\left( \varpi \right) }\right) dt_{j}\leq
\frac{1}{q_{\Gamma }(j)}\sum\limits_{_{_{\substack{ \varpi \in R(\Gamma )
\\ d_{1}\left( \varpi \right) =j}}}}\int_{0}^{T}\left\vert
B(t_{j}-t_{d_{2}(\varpi )})\right\vert ^{q_{\Gamma }(j)}dt_{j}\leq
\end{equation*}%
\begin{equation}
\leq 2\int_{0}^{T}\left\vert B(t_{j})\right\vert ^{q_{\Gamma }(j)}dt_{j}.
\label{6.25}
\end{equation}%
\bigskip If $q_{\Gamma }(i)=0,$ the integrals regarded to these variables,
after (\ref{6.25}), give a contribution in the form of multiplier $T$ in the
estimate (\ref{6.24}).

\begin{definition}
The level $j$ of a nonregular diagram $\Gamma \in L\backslash L^{\ast }$ is
said to be a donor, if $q_{\Gamma }(j)$ $\geq 1,$ \ and a strong donor, if $%
q_{\Gamma }(j)=l_{j}.$The level $j$ of a nonregular diagram $\Gamma \in
L\backslash L^{\ast }$ is said to be a recipient, if it is not donor, that
is, $q_{\Gamma }(j)=0.$
\end{definition}

Let $\rho _{sd\text{ }}$be a number of strong donor levels, and $\rho _{r}$
is a number of recipient levels. Note that $\rho _{sd\text{ }},\rho _{r}\geq
1,$ since the level $1$ is a strong donor one, while the level $p$ is a
recipient one.

Formulae (\ref{6.24}) \ and (\ref{6.25}) now imply
\begin{equation}
\left\vert I_{\Gamma }(J,T)\right\vert \leq 2^{p-\rho _{r}}\left\Vert
z\right\Vert ^{p}\left\Vert k(0)\right\Vert
^{p}T^{-p/2}\prod\limits_{j=1}^{p}\int_{0}^{T}\left\vert B(t)\right\vert
^{q_{\Gamma }(j)}dt.  \label{6.26}
\end{equation}

\bigskip Since $l_{j}\geq m\geq 2,$ under \textbf{B6}, for a strong donor
level $j$%
\begin{equation}
\int_{0}^{T}\left\vert B(t)\right\vert ^{l_{j}}dt\leq \int_{0}^{\infty
}[B(t)]^{2}dt<\infty .  \label{4.27}
\end{equation}

Thus, for the recipient levels ($q_{\Gamma }(j)$ $=0),$ and the strong donor
levels ($q_{\Gamma }(j)$ $=l_{j}$), we obtain%
\begin{equation}
\int_{0}^{T}\left\vert B(t)\right\vert ^{q_{\Gamma }(j)}dt\leq
c_{0}T^{1-z(j)},  \label{6.29}
\end{equation}%
\noindent where $z(j)=\frac{q_{\Gamma }(j)}{l_{j}},$ $c_{0}=\max \left(
1,\int_{0}^{\infty }B(t)^{2}dt\right) .$ Let now $0<q_{\Gamma }(j)<l_{j},$
that is, a level $j$ is a donor one, but not strong donor. Since $%
\int_{0}^{\infty }\left\vert B(t)\right\vert ^{l_{j}}dt<\infty ,$ for any $%
\epsilon >0,$ there exists $T_{\epsilon },$ such that $\int_{T_{\epsilon
}}^{\infty }\left\vert B(t)\right\vert ^{l_{j}}dt<\epsilon .$ Hence, it
follows from the H{\"o}lder inequality ($\frac{1}{p}=z(j)=\frac{q_{\Gamma
}(j)}{l_{j}},\frac{1}{q}=1-z(j)),$ that for sufficiently large $T$%
\begin{equation*}
\int_{0}^{T}\left\vert B(t)\right\vert ^{q_{\Gamma
}(j)}dt=\int_{0}^{T_{\epsilon }}\left\vert B(t)\right\vert ^{q_{\Gamma
}(j)}dt+\int_{T_{\epsilon }}^{T}\left\vert B(t)\right\vert ^{q_{\Gamma
}(j)}dt\leq
\end{equation*}%
\begin{equation*}
\leq C(\epsilon )+\left( \int_{T_{\epsilon }}^{T}\left[ \left\vert
B(t)\right\vert ^{q_{\Gamma }(j)}\right] ^{\frac{l_{j}}{q_{\Gamma }(j)}%
}dt\right) ^{z(j)}T^{1-z(j)}\leq
\end{equation*}%
\begin{equation}
\leq C(\epsilon )+\epsilon ^{z(j)}T^{1-z(j)}=o(T^{1-z(j)}).  \label{6.30}
\end{equation}

Denoting by $\mu =\frac{p}{2}-\sum\limits_{j=1}^{p}z(j),$ formulae (\ref%
{6.29}) \ and (\ref{6.30}) together with (\ref{6.26}) lead to
\begin{equation}
\left\vert I_{\Gamma }(J,T)\right\vert =O(T^{\mu }),\quad T\longrightarrow
\infty ,  \label{6.31}
\end{equation}%
\noindent if all the levels of $\Gamma $ are strong donor and recipient, and
\begin{equation}
\left\vert I_{\Gamma }(J,T)\right\vert =o(T^{\mu }),\quad T\longrightarrow
\infty ,  \label{6.32}
\end{equation}%
\noindent if $0<q_{\Gamma }(j)<l_{j},$ for some $j.$

Let us show that $\mu \leq 0.$ Choose an edge $\varpi \in R(\Gamma ),$ and
define the numbers $p_{1}(\varpi )$ and $p_{2}(\varpi )$ as the
cardinalities of levels $d_{1}(\varpi )$ and $d_{2}(\varpi )$ respectively.
Observe that $p_{1}(\varpi )\leq p_{2}(\varpi ),$ for any $\varpi \in
R(\Gamma ).$ Taking into account the definition of the $z(j),$ we obtain%
\begin{equation}
2\sum\limits_{j=1}^{p}z(j)=2\sum\limits_{j=1}^{p}\frac{q_{\Gamma }(j)}{l_{j}}%
=2\sum\limits_{\varpi \in R(\Gamma )}\frac{1}{p_{1}(\varpi )}\geq
\sum\limits_{\varpi \in R(\Gamma )}\left[ \frac{1}{p_{1}(\varpi )}+\frac{1}{%
p_{2}(\varpi )}\right] =p,  \label{6.33}
\end{equation}%
because the term $1/l_{i}$ appears exactly $l_{i}$ times among the summands $%
1/p_{1}(\varpi )$ and $1/p_{2}(\varpi ).$ The following inequality then
holds
\begin{equation}
\sum\limits_{1\leq i\leq p}z(i)\geq \frac{p}{2},\quad \text{or }\quad \mu
\leq 0,  \label{6.34}
\end{equation}%
where there is a strict inequality if $\Gamma $ contains an edge connecting
levels of different cardinality.

Thus, if in $\Gamma =\Gamma \left( l_{1},\dots ,l_{p}\right) \in \mathit{L}%
\backslash \mathit{L}^{\ast },$ $l_{1}\leq \dots \leq l_{p},$ there is an
edge between levels of different cardinalities, and all the levels are
strongly donor or recipient ones, then, from (\ref{6.31}),
\begin{equation}
\left\vert I_{\Gamma }(J,T)\right\vert \longrightarrow 0,\quad
T\longrightarrow \infty ,  \label{6.35}
\end{equation}%
\noindent while if there is level $j$ such that $0<q_{\Gamma }(j)<l_{j},$
then (\ref{6.35}) follows from (\ref{6.32}) and (\ref{6.34}).

We assume now that all edges of a non-regular diagram $\Gamma =\Gamma \left(
l_{1},\dots,l_{p}\right) \in \mathit{L}\backslash \mathit{L}^{\ast },$ $%
l_{1}\leq \dots\leq l_{p},$ connect the levels of the same cardinality. To
complete the proof one can use the following observations: Let $i$ be the
first upper recipient level. If it got an edge from the one (strongly donor)
level upper it, then the integral on variable $t_{j}$ in the right hand size
of (\ref{6.26}) can be estimated by a constant, while the integral on the
variable corresponding to the above strongly donor level can be estimated by
$T.$ Then, one can remove these levels from the consideration. Thus, we can
consider the moment of order $(p-2)$ (instead of order $p$). Since the
diagram $r$ is nonregular, one can continue the above procedure until the
case where the recipient level got edges from more than one donor level
upper it.

Let $i$ be the first upper recipient level that has edges from at least two
donor's levels $j$ and $k$ upper it, $j<k<i,$ and $k$ is the nearest to $i$
donor level. Level $k$ does not give all edges to $i$.

Let us change $k$ and $i$, and denote this permutation by $\tilde{\pi}.$
Then, $\tilde{\pi}\ (k)=i,\ \tilde{\pi}(i)=k,\ \tilde{\pi}(i)<\ \tilde{\pi}%
(k),$ and from the level $\widetilde{\pi }(i)$ to $\widetilde{\pi }(k)$ will
enter less than $l_{i}=l_{k}=l$ edges. Moreover, $q_{\tilde{\pi}\Gamma }(%
\tilde{\pi}(i))=q_{\tilde{\pi}\Gamma }(k)<l,$ since the only down edges from
$\tilde{\pi}(i)$ are those connecting $\tilde{\pi}(i)$ with $\tilde{\pi}(k).$

Let nonregular diagram $\Gamma =\Gamma (l_{1},\ldots ,l_{p}),\ l_{1}\leq
\cdots \leq l_{p},$ does not contain any donor level $j$ such that $%
q_{\Gamma }(j)<l_{j}.$ Then, the following dichotomy holds: either $\Gamma $
connects the levels of different cardinalities, or there exists a
permutation of a strongly donor (say, $k$-th), and recipient levels, such
that $0<q_{\tilde{\pi}\Gamma }(k)<l,$ where $l$ is the joint cardinality of
both levels. Thus, we have proven (\ref{6.21}), i.e., the following
statement holds:

\begin{theorem}
\label{th5} Under conditions \textbf{A1}-\textbf{A4}, \textbf{B1}-\textbf{B3
}and \textbf{B6}, the random vector (\ref{6.1}) converges in distribution,
as $T\longrightarrow \infty ,$ to the Gaussian vector $N(0,\Sigma ),$ where
\begin{equation}
\Sigma =2\pi \sum\limits_{k=m}^{\infty }\frac{C_{k}^{2}}{k!}\int_{\mathbb{R}%
}f^{(\ast k)}(\lambda )\boldsymbol{\mu }(d\lambda ,\theta ),  \label{6.36}
\end{equation}%
with $\boldsymbol{\mu }(d\lambda ,\theta )$ being the spectral measure of
the regression function, and $f^{(\ast k)}(\lambda )$ being the $k$th
self-convolution of s.d. under assumption {\bfseries A1}.
\end{theorem}

Now we are able to prove the asymptotic normality of the LSE .

\begin{theorem}
Assume that conditions \textbf{A1}-\textbf{A4},\textbf{B1}-\textbf{B6 }and
\textbf{C} hold. Then, the random vector $\hat{u}_{T}=d_{T}(\theta )(\hat{%
\theta}_{T}-\theta )$ converges, in distribution, to the Gaussian vector $%
N(0,\Sigma _{0}),$ as $T\longrightarrow \infty ,$ where
\begin{equation}
\Sigma _{0}=2\pi \sum\limits_{k=m}^{\infty }\frac{C_{k}^{2}}{k!}\left( \int_{%
\mathbb{R}}\boldsymbol{\mu }(d\lambda ,\theta )\right) ^{-1}\left( \int_{%
\mathbb{R}}f^{(\ast k)}(\lambda )\boldsymbol{\mu }(d\lambda ,\theta )\right)
\left( \int_{\mathbb{R}}\boldsymbol{\mu }(d\lambda ,\theta )\right) ^{-1}.
\label{6.37}
\end{equation}
\end{theorem}

\begin{proof}
In the notation of Sections 4 and 5, we obtain
\begin{equation}
L_{T}^{i}(u)=\int_{0}^{T}\left( G(\xi
(t))-\sum\limits_{l=1}^{q}v_{lT}(t,\theta )u_{l}\right)
v_{iT}(t,\theta )dt=0,\quad i=1,\ldots ,q, \label{6.38}
\end{equation}%
\noindent or equivalently,
\begin{eqnarray*}
\int_{0}^{T}G(\xi (t))v_{iT}(t,\theta )dt
&=&\int_{0}^{T}\sum\limits_{l=1}^{q}v_{lT}(t,\theta )u_{l}v_{iT}(t,\theta )dt, \\
\sum\limits_{l=1}^{q}\int_{0}^{T}v_{lT}(t,\theta )v_{iT}(t,\theta
)dtu_{l} &=&\sum\limits_{l=1}^{q}J_{il,T}(\theta
)u_{l}=\int_{0}^{T}G(\xi (t))v_{iT}(t,\theta )dt.
\end{eqnarray*}

Thus, we have a system of equations regarded to $u:$ $J_{T}(\theta
)u=d_{T}^{-1}(\theta )\int\limits_{0}^{T}G(\xi (t))\nabla g(t,\theta
)dt,$ or $\widetilde{u}_{T}=(\widetilde{u}_{1},\ldots
,\widetilde{u}_{q})^{\prime }=\Lambda _{T}(\theta
)\int\limits_{0}^{T}G(\xi (t))d_{T}^{-1}(\theta )\nabla g(t,\theta
)dt=\Lambda _{T}(\theta )\zeta _{T},$
 where $\Lambda _{T}(\theta )=J_{T}^{-1}(\theta ),$ $d_{T}(\theta
)=diag(d_{iT}(\theta ))_{i=1}^{q},$ $\nabla g(t,\theta )=\left(
g_{1}(t,\theta ),\ldots ,g_{q}(t,\theta )\right) ^{\prime }.$

From  Theorem \ref{th5}, the vector $$\tilde{u}_{T}=\Lambda
_{T}(\theta )\left( \int_{0}^{T}\varepsilon (t)v_{iT}(t,\theta
)dt\right) _{i=1}^{q}$$ \noindent is asymptotically normal. To
compute the limiting covariance matrix, we note that the covariance
matrix of the vector $\tilde{u}_{T}$  has the form $ \Sigma
_{0T}=\Lambda _{T}(\theta )\sigma _{T}^{2}(\theta )\Lambda
_{T}(\theta ),$ where $\sigma _{T}^{2}(\theta )$ is covariance
matrix of the vector $\zeta _{T}.$ As $T\longrightarrow \infty ,$
\begin{eqnarray}
\Sigma _{0T} &=&2\pi \sum\limits_{k=m}^{\infty }\frac{C_{k}^{2}}{k!}\left(
\int_{\mathbb{R}}\boldsymbol{\mu }_{T}(d\lambda ,\theta )\right) ^{-1}\left( \int_{%
\mathbb{R}}f^{(\ast k)}(\lambda )\boldsymbol{\mu }_{T}(d\lambda
,\theta )\right) \left(
\int_{\mathbb{R}}\boldsymbol{\mu } _{T}(d\lambda ,\theta )\right) ^{-1}  \label{6.12b} \\
&\longrightarrow & 2\pi \sum\limits_{k=m}^{\infty }\frac{C_{k}^{2}}{k!}%
\left( \int_{\mathbb{R}}\boldsymbol{\mu } (d\lambda ,\theta )\right) ^{-1}\left( \int_{%
\mathbb{R}}f^{(\ast k)}(\lambda )\boldsymbol{\mu } (d\lambda ,\theta )\right) \left( \int_{%
\mathbb{R}}\boldsymbol{\mu } (d\lambda ,\theta )\right) ^{-1}=\Sigma
_{0}. \nonumber
\end{eqnarray}

We need to prove that the d.f. $F_{T}(y,\theta )$ of the vector
$\hat{u}_{T}=d_{T}(\theta )(\hat{\theta}_{T}-\theta )$ converges to
the Gaussian  d.f. $\Phi _{0,\Sigma _{0}}(y)$ as $T\longrightarrow
\infty .$

Then, we will show that, for any $r>0,$
\begin{equation}
\Delta _{T}(r)=P\left\{ \left\Vert \hat{u}_{T}-\tilde{u}_{T}\right\Vert
>r\right\} \longrightarrow 0,\ T\longrightarrow \infty .  \label{6.13b}
\end{equation}

Denote the event $A_{T}=\left\{ \tilde{u}_{T}\in
V^{c}(R-r)\right\},$ where $R$ is such that, for $T>T_{0},$
$P(\overline{A}_{T})<\frac{\epsilon }{2},$ for a fixed $\epsilon
>0.$  This follows from the asymptotic normality of $\tilde{u}_{T}.$
Introduce one more event $B_{T}=\left\{ \sup_{u\in
V^{c}(R)}\left\Vert \Lambda _{T}(\theta )\left( \psi_{T}(u)-L
_{T}(u)\right) \right\Vert \right\} \leq r.$ From  Theorem
\ref{uniquiness}, we obtain that, for $T>T_{0},$
\begin{eqnarray*}
P(B_{T}) &=& P\left\{ \sup_{u\in V^{c}(R)}\left\Vert \Lambda
_{T}(\theta )\left( \psi _{T}(u)-L_{T}(u)\right) \right\Vert
>r\right\} \leq P\left\{ \lambda _{\max }(\Lambda _{T}(\theta
))\sup_{u\in V^{c}(R)}\left\Vert \psi
_{T}(u)-L_{T}(u)\right\Vert >r\right\} \\
&=&P\left\{ \frac{1}{\lambda _{\min }(J_{T}(\theta ))}\sup_{u\in
V^{c}(R)}\left\Vert \psi _{T}(u)-L_{T}(u)\right\Vert >r\right\} \leq
P\left\{ \frac{1}{\lambda _{\ast }}\sup_{u\in V^{c}(R)}\left\Vert
\psi
_{T}(u)-L_{T}(u)\right\Vert >r\right\} \\
&=& P\left\{ \sup_{u\in V^{c}(R)}\left\Vert \psi
_{T}(u)-L_{T}(u)\right\Vert
>\lambda _{\ast }r\right\} \leq \frac{\epsilon }{3}.
\end{eqnarray*}

Introduce also the event $\ C_{T}=\left\{ \text{LSE
}\hat{u}_{T}\text{ is  unique solution of the system
(\ref{4.9})}\right\}.$ From Theorem 3, consider  $T>T_{0}$ such that
$\ P\left\{ \overline{C}_{T}\right\} \leq \frac{\epsilon }{3}.$
Thus, for $T>T_{0},$
\begin{equation}
P(A_{T}\cap B_{T}\cap C_{T})\geq 1-\epsilon . \label{6.14b}
\end{equation}
\noindent Then,
\begin{eqnarray*}
\Lambda _{T}(\theta )L_{T}(u) &=&\Lambda _{T}(\theta )\left(
\int_{0}^{T}\varepsilon (t)v_{iT}(t,\theta )dt\right)
_{i=1}^{q}-\Lambda _{T}(\theta )\left(
\sum\limits_{l=1}^{q}u_{l}\int_{0}^{T}v_{lT}(t,\theta )v_{iT}(t,\theta )dt\right) _{i=1}^{q} \\
&=&\tilde{u}_{T}-\Lambda _{T}(\theta )\left(
\sum\limits_{l=1}^{q}u_{l}J_{il,T}(\theta )\right) _{i=1}^{q}=\tilde{u}%
_{T}-u.
\end{eqnarray*}

For $u\in V^{c}(R),$ under event $A_{T}\cap B_{T}\cap C_{T},$ we
have
\begin{eqnarray*}
\left\Vert u+\Lambda _{T}(\theta )\psi _{T}(u)\right\Vert &=&
\left\Vert u+\Lambda
_{T}(\theta )(\psi _{T}(u)-L_{T}(\theta ))+\Lambda _{T}(\theta )L_{T}(u)\right\Vert \\
&=&\left\Vert u+\tilde{u}_{T}-u+\Lambda _{T}(\theta )(\psi
_{T}(u)-L_{T}(\theta ))\right\Vert \leq \left\Vert
\tilde{u}_{T}\right\Vert +\left\Vert \Lambda _{T}(\theta )(\psi
_{T}(u)-L_{T}(\theta )\right\Vert \nonumber\\ &\leq & R-r+r=R,
\end{eqnarray*}%
\noindent that is, $$F_{T}(u)=u+\Lambda _{T}(\theta )(\psi
_{T}(u)):\ V^{c}(R)\rightarrow V^{c}(R)$$ is a continuous map.

To prove (\ref{6.13b}) we will apply   Fix Point Brouwer Theorem
(Milnor (1965), p. 14). Specifically, if $F:\ V^{c}(R)\rightarrow
V^{c}(R)$ is a continuous map, then, there exists $x_{0}\in
V^{c}(R)$ such that $F(x_{0})=x_{0}.$
From Brouwer Theorem, there exists $u_{T}^{0}\in \ V^{C}(R)$ such that $%
F_{T}(u_{T}^{0})=u_{T}^{0},$ and hence, $\psi _{T}(u_{T}^{0})=0,$ since $%
\Lambda _{T}(\theta )$ is non degenerated.

Under $C_{T},$ the normed LSE $\hat{u}_{T}$ is the unique solution
to the equation $$\psi _{T}(u)=0,\quad  u\in V^{c}(R).$$

Thus, $A_{T}\cap B_{T}\cap C_{T}\subset \left\{ \hat{u}_{T}\in
V^{c}(R)\right\}$ and $P\left\{ \hat{u}_{T}\in V^{c}(R)\right\} \geq
1-\epsilon .$ From (\ref{6.14b}), we get
\begin{eqnarray}
1-\epsilon &\leq &P\left\{ \left\{ \hat{u}_{T}\in V^{c}(R)\right\}
\cap B_{T}\right\} =P\left\{ \left\{  \hat{u}_{T}\in
V^{c}(R)\right\} \cap \left\{ \sup_{u\in V^{c}(R)}\left\Vert \Lambda
_{T}(\theta )(\psi
_{T}(u)-L_{T}(u))\right\Vert \leq r\right\} \right\}  \nonumber\\
&\leq & P\left\{ \left\Vert \Lambda _{T}(\theta )(\psi
_{T}(\hat{u})-L_{T}(\hat{u}))\right\Vert \leq r\right\} =P\left\{
\left\Vert \Lambda _{T}(\theta )L_{T}(\widehat{u})\right\Vert \leq
r\right\} =P\left\{ \left\Vert
\tilde{u}_{T}-\hat{u}_{T}\right\Vert \leq r\right\} .\nonumber\\
\label{6.15c}
\end{eqnarray}%
Therefore, (\ref{6.13b}) follows from (\ref{6.15c}).

Let us consider the notation
\begin{equation*}
\Pi (-\infty ;y\pm \vec{\epsilon})=(-\infty ;y_{1}\pm \epsilon
)\times \cdots \times (-\infty ;y_{q}\pm \epsilon ),\quad \epsilon
\geq 0.
\end{equation*}%
For the d.f. $F_{T}(y,\theta )=P\left\{ \tilde{u}_{T}\in \Pi
(-\infty ;y)\right\},$ we obtain from (\ref{6.13b}),
\begin{equation}
F_{T}(y,\theta )\geq P\left\{ \tilde{u}_{T}\in \Pi (-\infty ;y-\vec{%
\epsilon})\right\} -\Delta _{T}(\epsilon ),\
F_{T}(y,\theta )\leq P\left\{ \tilde{u}_{T}\in \Pi (-\infty ;y-\vec{%
\epsilon})\right\} +\Delta _{T}(\epsilon ),  \label{6.18b}
\end{equation}%
\noindent for any $y\in \mathbb{R}^{q},$ and $\epsilon >0.$ We know
that
\begin{equation}
\left\vert P\left\{ \tilde{u}_{T}\in \Pi (-\infty ;y\pm \vec{\varepsilon}%
)\right\} -\Phi _{0,\gamma (\theta )}(y\pm
\vec{\varepsilon})\right\vert \longrightarrow 0,\quad
T\longrightarrow \infty . \label{6.19b}
\end{equation}%
Let $\phi (y,\theta )$ be the probability density function of a
Gaussian  random variable with d.f.
 $\Phi_{0,\gamma (\theta )}(y).$ Since $\lambda _{\min }(\gamma (\theta
))=\underline{\lambda }>0,$ $\lambda _{\max }(\gamma (\theta
))=\bar{\lambda}<\infty ,$ then
\begin{equation*}
\phi (y,\theta )\leq (2\pi \underline{\lambda })^{-q/2}\exp \left\{ \frac{%
-\left\Vert y\right\Vert ^{2}}{2\bar{\lambda}}\right\} =\nu (\left\Vert
y\right\Vert ).
\end{equation*}

For $\mathcal{A}\in \mathcal{B}^{q},$ with $\mathcal{B}^{q}$ being
the $\sigma $-algebra of Borel sets of $\mathbb{R}^{q},$ and for
$\epsilon >0,$ let $$A_{\epsilon }=\left\{ x\in \mathbb{R}^{q}:\inf_{y\in \mathcal{A}%
}\left\Vert x-y\right\Vert <\epsilon \right\} ,\quad  A_{-\epsilon }=%
\mathbb{R}^{q}\backslash (\mathbb{R}^{q}\backslash A)_{\epsilon }.$$
If $A=\Pi (-\infty ;y),$ then, $A_{-\epsilon }=\Pi (-\infty ;y-\vec{%
\epsilon}),\ \Pi (-\infty ;y+\vec{\epsilon})_{-\epsilon
}=\mathbb{R}^{q}\backslash\Pi (-\infty ;y)=A^{c}$. We will apply
Theorem {\S}3 of Bhattacharya and Ranga Rao (1976).

\begin{lemma}
\label{lemcof}
Let $\nu $ be a non-negative differential function on $%
[0,\infty ),$ such that
\begin{itemize}
\item[(1)] $b=\int_{0}^{\infty }\left\vert \nu ^{\prime }(\lambda
)\right\vert \lambda ^{q-1}d\lambda <\infty ;$

\item[(2)] $\lim_{\lambda \rightarrow \infty }\nu (\lambda )=0.$
\end{itemize}

Then for any convex $C\in \mathcal{B}^{q}$ and given $\epsilon
,\delta
>0 $, we have%
\begin{equation*}
\int_{C_{\epsilon }\backslash C_{-\delta }}\nu \left(\left\Vert \lambda
\right\Vert\right)  d\lambda \leq b\left( \frac{2\pi ^{q/2}}{\Gamma (\frac{q}{2})}%
\right) (\epsilon +\delta ).
\end{equation*}
\end{lemma}
From Lemma \ref{lemcof}, for any $\psi \neq 0,$ we have%
\begin{equation*}
\left\vert \Phi _{0,\gamma (\theta )}(y)-\Phi _{0,\gamma (\theta )}(y+\vec{%
\psi})\right\vert =\int_{\Pi }\phi (y,\theta )dy\leq b\left(
\frac{2\pi ^{q/2}}{\Gamma (\frac{q}{2})}\right) \left\vert \psi
\right\vert ,
\end{equation*}%
where $\Pi =\left\{
\begin{array}{cc}
\Pi (-\infty ,y+\vec{\psi})\backslash A^{\small c}, & \psi >0, \\
A\backslash A_{\psi },\ \ \ \ \ \ \ \ \ \ \ \ \ \ \ \ \  & \psi <0.%
\end{array}%
\right.$

For any $y\in \mathbb{R}^{q},$ and $\epsilon >0,$
\begin{eqnarray*}
F_{T}(y,\theta )-\Phi _{0,\gamma (\theta )}(y) &\leq &\Delta
_{T}(\epsilon )+P\left\{ \tilde{u}_{T}\in \Pi (-\infty ;y+\vec{\epsilon%
})\right\} -\Phi _{0,\gamma (\theta )}(y) \\
&\leq &\Delta _{T}(\epsilon )+\left\vert P\left\{ \tilde{u}_{T}\in
\Pi (-\infty ;y+\vec{\epsilon})\right\} -\Phi _{0,\gamma (\theta
)}(y)\right\vert \\
&\leq &\Delta _{T}(\epsilon )+\left\vert P\left\{ \tilde{u}_{T}\in
\Pi
(-\infty ;y+\vec{\epsilon})\right\} -\Phi _{0,\gamma (\theta )}(y+\vec{%
\epsilon})\right\vert +\left\vert \Phi _{0,\gamma (\theta )}(y+\vec{%
\epsilon})-\Phi _{0,\gamma (\theta )}(y)\right\vert ;
\end{eqnarray*}

\begin{eqnarray*}
\Phi _{0,\gamma (\theta )}(y)-F_{T}(y,\theta ) &\leq &\Delta
_{T}(\epsilon )-P\left\{ \tilde{u}_{T}\in \Pi (-\infty ;y-\vec{\epsilon%
})\right\} +\Phi _{0,\gamma (\theta )}(y) \\
&\leq &\Delta _{T}(\epsilon )+\left\vert \Phi _{0,\gamma (\theta
)}(y)-P\left\{ \tilde{u}_{T}\in \Pi (-\infty
;y-\vec{\epsilon})\right\}
\right\vert \\
&\leq &\Delta _{T}(\epsilon )+\left\vert \Phi _{0,\gamma (\theta )}(y-%
\vec{\epsilon})-P\left\{ \tilde{u}_{T}\in \Pi (-\infty ;y-\vec{\epsilon%
})\right\}\right\vert +\left\vert \Phi _{0,\gamma (\theta )}(y)-\Phi
_{0,\gamma (\theta )}(y-\vec{\epsilon})\right\vert .
\end{eqnarray*}

Therefore, we have
\begin{equation*}
\left\vert F_{T}(y,\theta )-\Phi _{0,\gamma (\theta )}(y)\right\vert
\longrightarrow 0,\ T\longrightarrow \infty .
\end{equation*}
Thus, Theorem 5 is proven.
\end{proof}

\section{\textbf{Asymptotic normality of the LSE of the parameters of
trigonometric regression }}

The asymptotic Gaussian distribution of the LSE in the Walker sense of the
regression function (\ref{nonregr2}) is established in the following result.

\begin{theorem}
\label{th6} Under conditions \textbf{A1}-\textbf{A4}, the LSE in the Walker
sense of the function (\ref{nonregr2}) of unknown parameter is
asymptotically normal, that is, the vector
\begin{equation*}
\left( T^{1/2}(\hat{A}_{kT}-A),T^{1/2}(\hat{B}_{kT}-B),T^{3/2}(\hat{\varphi}%
_{kT}-\varphi ),\ k=1,\ldots ,N\right)
\end{equation*}
\noindent converges weakly to the multidimensional normal vector $%
N_{3N}(0,\Gamma ),$ where the matrix $\Gamma >0$ is of the form $\Gamma
=diag\left( \Gamma _{k}\right) _{k=1}^{N},$ with
\begin{equation}
\Gamma _{k}=\frac{4\pi }{A_{k}^{2}+B_{k}^{2}}\sum_{j=m}^{\infty }\frac{%
C_{j}^{2}}{j!}f^{(\ast j)}(\varphi _{k})\left(
\begin{array}{ccc}
A_{k}^{2}+B_{k}^{2} & -3A_{k}B_{k} & -6B_{k} \\
-3A_{k}B_{k} & A_{k}^{2}+B_{k}^{2} & 6A_{k} \\
-6B_{k} & 6A_{k} & 12%
\end{array}%
\right) .  \label{matrixgamma}
\end{equation}
\noindent Here, $f^{(\ast j)}(\lambda ),$ $\lambda \in \mathbb{R},$ is the $%
j $-th convolution of the spectral density given under assumption {\bfseries %
A1}.
\end{theorem}

The spectral measures of the trigonometric regression were investigated, for
example, by Whittle (1952), Walker (1973) and Ivanov (1980) (see also the
monograph by Quinn and Hannan, 2001). Theorem \ref{th6} follows from the
results of Section 5 by direct computations. Indeed, for the nonlinear
function (\ref{nonregr2}), the spectral measure $\mu (d\lambda ,\theta
)=diag\left( \tilde{\Gamma}_{k}\right) _{k=1}^{N},$ where%
\begin{equation*}
\tilde{\Gamma}_{k}=\left(
\begin{array}{ccc}
\delta _{k} & i\rho _{k} & \bar{\beta}_{k} \\
i\rho _{k} & \delta _{k} & \bar{\gamma}_{k} \\
\beta _{k} & \gamma _{k} & \delta _{k}%
\end{array}%
\right) ,\quad \beta _{k}=\frac{\sqrt{3}(B_{k}\delta _{k}+iA_{k}\rho _{k})}{2%
\sqrt{A_{k}^{2}+B_{k}^{2}}},\quad \gamma _{k}=\frac{\sqrt{3}(-A_{k}\delta
_{k}-iB_{k}\rho _{k})}{2\sqrt{A_{k}^{2}+B_{k}^{2}}},
\end{equation*}

\noindent with the measure $\delta _{k}=\delta _{k}(d\lambda ),$ and the
signed measure $\rho _{k}=\rho _{k}(d\lambda )$ being located at the points $%
\pm \varphi _{k},$ $k=1,\dots ,N.$ Here, $\delta _{k}\left( \left\{ \pm
\varphi _{k}\right\} \right) =\frac{1}{2};$ $\rho _{k}\left( \left\{ \pm
\varphi _{k}\right\} \right) =\pm \frac{1}{2},$ $k=1,\ldots ,N.$ Thus, $%
\Sigma _{0}=diag\left( \Sigma _{0k}\right) _{k=1}^{N},$ where
\begin{equation*}
\Sigma _{0k}(\theta )=2\pi \sum_{j=m}^{\infty }\frac{C_{j}^{2}}{j!}f^{(\ast
j)}(\varphi _{k})\left(
\begin{array}{ccc}
1 & 0 & \frac{B_{k}}{\sqrt{\frac{3}{4}(A_{k}^{2}+B_{k}^{2})}} \\
0 & 1 & \frac{-A_{k}}{\sqrt{\frac{3}{4}(A_{k}^{2}+B_{k}^{2})}} \\
\frac{B_{k}}{\sqrt{\frac{3}{4}(A_{k}^{2}+B_{k}^{2})}} & \frac{-A_{k}}{\sqrt{%
\frac{3}{4}(A_{k}^{2}+B_{k}^{2})}} & 1%
\end{array}%
\right) ^{-1},
\end{equation*}%
\noindent and direct computations complete the proof.

\section{Remarks on some future development}
To make  Theorem \ref{th5} operational, some estimation results to
approximate the limiting variance in (\ref{6.36}) should be needed.
In general this problem deserves a separate publication, but in
short one can use the following arguments.

Consider the block-diagonal covariance matrix $\Gamma $ in equation (\ref%
{matrixgamma}). Let us take for its blocks $\Gamma_{k}$ their statistical
estimators $\widehat{\Gamma}_{k}$ substituting into $\Gamma_{k}$ the LSE $%
(A_{kT}, B_{kT}, \varphi_{kT})$ instead of unknown parameters $%
(A_{k},B_{k},\varphi_{k}).$

\begin{theorem}
Under conditions \textbf{A1-A4},
$\widehat{\Gamma}_{k}\overset{P}{\longrightarrow }\Gamma_{k},$
$T\longrightarrow \infty,$ $k=1,\dots,N.$
\end{theorem}

\begin{proof}
Using the notation $B_{m}=\int_{-\infty }^{\infty }|B(t)|^{m}dt<\infty,$  for $j\geq m,$ obviously,
\begin{equation}
|f^{(*j)}(\widehat{\varphi}_{kT})|= \frac{1}{2\pi}\left| \int_{-\infty }^{\infty }B^{j}(t)\cos\left( \widehat{\varphi
}_{kT}t\right)dt\right|\leq \frac{1}{2\pi }B_{m}; \label{eq81} \end{equation}

\begin{eqnarray}
\left|f^{(*j)}(\widehat{\varphi}_{kT})-f^{(*j)}(\varphi_{kT})\right|&=& \frac{1}{2\pi}\left|\int_{-\infty }^{\infty
}B^{j}(t)[\cos\left( \widehat{\varphi }_{kT}t\right)-\cos\left( \varphi _{kT}t\right)]dt\right|\nonumber\\
&\leq & \frac{1}{2\pi}\int_{-T}^{T}|B(t)|^{j}|\cos\left( \widehat{\varphi }_{kT}t\right)-\cos\left( \varphi
_{kT}t\right)|dt+\frac{2}{T}\int_{T}^{\infty }|B(t)|^{j}dt\nonumber\\
&\leq & \frac{B_{m}}{2\pi }T|\widehat{\varphi}_{kT}-\varphi_{k}|+\frac{2}{T}\int_{T}^{\infty }|B(t)|^{m}dt.
 \label{eq82}
\end{eqnarray}
Inequality (\ref{eq82}) shows that for $j\geq m,$ $k=1,\dots,N,$
\begin{equation}
f^{(*j)}(\widehat{\varphi}_{kT})\overset{P}{\longrightarrow }
f^{(*j)}(\varphi_{kT}),\quad T\longrightarrow \infty. \label{eq83}
\end{equation}

For convergent series $\sum_{j=1}^{\infty }\frac{C_{j}^{2}}{j!},$ and any fixed $\varepsilon >0,$ let
$n_{0}=n_{0}(\varepsilon )$ be such a number that
\begin{equation}
\sum_{j=m}^{\infty
}\frac{C_{j}^{2}}{j!}f^{(*j)}(\widehat{\varphi}_{kT})=\sum_{j=m}^{n_{0}}\frac{C_{j}^{2}}{j!}f^{(*j)}(\widehat{\varphi}_{kT})+
\sum_{j=n_{0}+1}^{\infty }\frac{C_{j}^{2}}{j!}f^{(*j)}(\widehat{\varphi}_{kT})=\Sigma_{1}+\Sigma_{2},\label{eq84}
\end{equation}
\noindent and according to (\ref{eq81}) $\Sigma_{2}\leq \frac{1}{2\pi}B_{m}\varepsilon .$

On the other hand, due to (\ref{eq83}), as $T\rightarrow \infty,$
$$\Sigma_{1}=\sum_{j=m}^{n_{0}}\frac{C_{j}^{2}}{j!}f^{(*j)}(\widehat{\varphi }_{kT})\longrightarrow
\sum_{j=m}^{n_{0}}\frac{C_{j}^{2}}{j!}f^{(*j)}(\varphi_{k}).$$

Therefore, as $T\longrightarrow \infty,$
$$4\pi \sum_{j=m}^{\infty }\frac{C_{j}^{2}}{j!}f^{(*j)}(\widehat{\varphi }_{kT})\overset{P}{\longrightarrow }4\pi \sum_{j=m}^{\infty }\frac{C_{j}^{2}}{j!}f^{(*j)}(\varphi
_{kT}).$$ The theorem is then proved due to consistency of $A_{kT},$
$B_{kT},$ $k=1,\dots,N.$

\end{proof}

The study of further properties of the proposed covariance matrix
$\Gamma $ estimator is a more difficult problem
          and we address it to subsequent publications

\begin{remark}
In this paper, we consider the continuous time stochastic processes and
observations which is more suitable framework for this classical statistical
problem. However, similar results can be obtained for a discrete
observations $x_{t},$ $t\in \{0,1,\dots ,T-1\}$ in the model (\ref{nonregr1}%
). We only need to replace $\int_{0}^{T}...$ by $\sum_{t=0}^{T-1}$ ... in
some steps, and instead of s.d. $f(\lambda ),$ $\lambda \in \mathbb{R},$ of
the stochastic process with continuous time, we have to use spectral density
of discretized process $f_{d}(\lambda )=\sum_{k=-\infty }^{\infty }f(\lambda
+2k\pi ),$ $\lambda \in \Lambda =(-\pi ,\pi ),$ and its convolutions. Also
in the integrals in the spectral domain we have to replace $\mathbb{R}^{k}$
into $\Lambda ^{k}.$

On the other hand, we can use different numerical procedures to solve a
problem related to discrete/continuous observations (for example, smoothing,
etc.).
\end{remark}

\section{Final Comments}

This paper addresses the problem of consistency, uniqueness and Gaussian
limit distribution of the LSE parameter estimate, in the Walker sense, for
the non-linear regression model (\ref{nonregr1}), where the regression
function has atomic spectral measure. This kind of regression actually
constitutes an active research area, due to the existence of several open
problems and applications. Note that, although here we have considered the
parameter range $\alpha =\min_{l=0,\dots ,\kappa}\alpha _{l}>1/m,$ our
conjecture is that the Gaussian limit results hold for
\begin{equation*}
\alpha _{l}\in (0,1),\quad l=0,\dots ,\kappa.
\end{equation*}
The proof of this conjecture will introduce a general scenario where most of
the limit results for random fields with singular spectra (see, for
instance, Taqqu (1975, 1979); Dobrushin and Major (1979); Nualart and
Peccati (2005); and the references therein) can be obtained as particular
cases. New limit results will be required, in the case where the two
spectra, the limit regression spectral measure and the spectrum of the
Gaussian random field generating the error term, can be overlapped. In this
case, different normalizing factors should be considered, leading to
different limiting distributions, depending on the common set of spectral
singularities co-existing in the regression and error spectra. This case
related to the resonance phenomenon will be investigated in subsequent
papers, where scaling factors will play a crucial role in the attainment of
new limit distributions, and in definition of robust estimates.

\subsection{Acknowledgements}

{\normalsize N.N. Leonenko and M.D. Ruiz-Medina partially supported by grant
of the European commission PIRSES-GA-2008-230804 (Maric Curie), projects
MTM2009-13393 and MTM2012-32674 of the DGI, and P09-FQM-5052 of the
Andalousian CICE, Spain, and the Australian Research Council grants
A10024117 and DP 0345577}

\section*{Appendices}

In the proofs of the Theorems 2 and 3, we use some ideas from Ivanov (1980,
1997, 2010 ), and Ivanov and Leonenko (1989, 2004, 2008, 2009).

\subsection*{Appendix 1}

Before the proof of Theorem 2, we formulate the following result on the
s.v.f.'s.

\begin{lemma}
\label{lem5} Let $\eta \geq 0$ be a given real number, and let function $%
f(t,s)$ being defined on $(0,\infty )\times (0,\infty )$ such that the
integral%
\begin{equation*}
\int\limits_{0}^{\beta }\int\limits_{0}^{\beta }f(t,s)\frac{dt\,ds}{%
\left\vert t-s\right\vert ^{\eta }}
\end{equation*}%
converges for \ some $\beta $ from $(0,\infty ).$ Let $L$ \ be a s.v.f. Then
for $\eta >0$%
\begin{equation*}
\int\limits_{0}^{\beta }\int\limits_{0}^{\beta }f(t,s)L(T\left\vert
t-s\right\vert )\frac{dt\,ds}{\left\vert t-s\right\vert ^{\eta }}\underset{%
T\longrightarrow \infty }{\sim }L(T)\int\limits_{0}^{\beta
}\int\limits_{0}^{\beta }f(t,s)\frac{dt\,ds}{\left\vert t-s\right\vert
^{\eta }},
\end{equation*}%
where $a(T)\sim b(T)$ means that $\lim_{T\rightarrow \infty
}\{a(T)/b(T)\}=1. $

If $\eta =0,$ this relation is valid when the function $L$ is nondecreasing
on semi-axis $(0,\infty ).$
\end{lemma}

The proof of Lemma 1 is similar to the proof of Theorem 2.7 in the book of
Seneta (1976).

\bigskip

\noindent \textbf{Proof of Theorem 2.} For any $i=1,\ldots ,q,$ one can
write the following identities:%
\begin{equation*}
\psi _{T}^{i}(u)-L_{T}^{i}(u)=
\end{equation*}%
\begin{equation*}
=\int\limits_{0}^{T}\varepsilon (t)\frac{h_{i}(t,u)}{d_{iT}(\theta )}%
dt+\int\limits_{0}^{T}H(t;0,u)\frac{h_{i}(t,u)}{d_{iT}(\theta )}%
dt-\int\limits_{0}^{T}\varepsilon (t)\frac{g_{i}(t,\theta )}{d_{iT}(\theta )}%
dt+\int\limits_{0}^{T}\frac{g_{i}(t,\theta )}{d_{iT}(\theta )}\sum_{l=1}^{q}%
\frac{g_{l}(t,\theta )}{d_{lT}(\theta )}u_{l}dt
\end{equation*}

\begin{equation*}
=\int\limits_{0}^{T}\varepsilon (t)\frac{H_{i}(t;u,0)}{d_{iT}(\theta )}%
dt+\int\limits_{0}^{T}H(t;0,u)\frac{H_{i}(t;u,0)}{d_{iT}(\theta )}%
dt+\int\limits_{0}^{T}\frac{g_{i}(t,\theta )}{d_{iT}(\theta )}\left[
H(t;0,u)+\sum_{l=1}^{q}\frac{g_{l}(t,\theta )}{d_{lT}(\theta )}u_{l}\right]
dt
\end{equation*}

\begin{equation}
=I_{1}(u)+I_{2}(u)+I_{3}(u).  \label{A.2}
\end{equation}

For fixed $u\in V^{c}(R)$, we get
\begin{equation*}
\mathit{E}I_{1}^{2}(u)=\int\limits_{0}^{T}\int\limits_{0}^{T}cov(\varepsilon
(t),\varepsilon (s))\frac{H_{i}(t;u,0)}{d_{iT}(\theta )}\frac{H_{i}(s;u,0)}{%
d_{iT}(\theta )}dtds\leq
\end{equation*}

\begin{equation}
\leq \sup_{t\in \left[ 0,T\right] }\left\{ \frac{H_{i}^{2}(t,u,0)}{%
d_{iT}^{2}(\theta )}\int\limits_{0}^{T}\int\limits_{0}^{T}\left\vert
E\varepsilon (t)\varepsilon (s)\right\vert dtds\right\}.  \label{A.1}
\end{equation}%
Under condition {\bfseries B3}, and using the finite difference formula, we
obtain the following estimates:
\begin{eqnarray}
\sup_{t\in \left[ 0,T\right] }\frac{\left\vert H_{i}(t;u,0)\right\vert }{%
d_{iT}(\theta )} &\leq &\sup_{t\in \left[ 0,T\right] }\left\{
\sum_{l=1}^{q}\sup_{u\in V^{c}(R)}\frac{\left\vert h_{il}(t;u)\right\vert }{%
d_{iT}(\theta )d_{lT}(\theta )}\left\vert u_{l}\right\vert \right\}  \notag
\\
&=&\sup_{t\in \left[ 0,T\right] }\left\{ \sum_{l=1}^{q}\sup_{u\in V^{c}(R)}%
\frac{\left\vert h_{il}(t;u)\right\vert }{d_{il,T}(\theta )}\cdot \frac{%
d_{il,T}(\theta )}{d_{iT}(\theta )d_{lT}(\theta )}\left\vert
u_{l}\right\vert \right\} \leq R\left( \sum_{l=1}^{q}k_{il}(R)\tilde{k}%
_{il}\right) T^{-1}.  \notag
\end{eqnarray}

Then,
\begin{equation*}
EI_{1}^{2}(u)\leq \left( \sum_{l=1}^{q}k_{u}(R)\tilde{k}_{il}\right)
^{2}R^{2}T^{-2}\int\limits_{0}^{T}\int\limits_{0}^{T}\left\vert E\varepsilon
(t)\varepsilon (s)\right\vert dtds.
\end{equation*}%
We will show that%
\begin{equation}
\frac{1}{T^{2}}\int\limits_{0}^{T}\int\limits_{0}^{T}\left\vert E\varepsilon
(t)\varepsilon (s)\right\vert dtds\longrightarrow 0,\quad T\longrightarrow
\infty .  \label{A.3}
\end{equation}%
From conditions \textbf{A1}-\textbf{A3},
\begin{equation}
\left\vert E\varepsilon (t)\varepsilon (s)\right\vert =\sum_{k=m}^{\infty }%
\frac{C_{k}^{2}}{k!}B^{k}(t-s)\leq EG^{2}(\xi (0))\left\vert
B(t-s)\right\vert^{m} .  \label{A.4}
\end{equation}

Thus, to prove (\ref{A.3}) we need to show that

\begin{equation}
\frac{1}{T^{2}}\int\limits_{0}^{T}\int\limits_{0}^{T}\left\vert
B(t-s)\right\vert^{m} dtds\longrightarrow 0,\quad T\longrightarrow \infty .
\label{A.5}
\end{equation}

This is straightforward, for $\alpha \geq 1,\alpha =\min (\alpha _{0},\ldots
,\alpha _{\kappa}).$ To prove (\ref{A.5}) for $\alpha <1,$ one can use the
inequality:
\begin{equation*}
\left\vert B(t)\right\vert \leq (1+t^{2})^{-\alpha /2}=B_{0}(t).
\end{equation*}

Thus, by the substitutions: $t=t^{\ast }T,$ $s=s^{\ast }T,$ we have that the
left hand side of (\ref{A.5}) is bounded by
\begin{eqnarray}
\frac{1}{T^{2}}\int\limits_{0}^{T}\int\limits_{0}^{T}B_{0}(t-s)dtds
&=&\int\limits_{0}^{1}\int\limits_{0}^{1}B_{0}(T(t^{\ast }-s^{\ast
}))dt^{\ast }ds^{\ast }  \label{A.6} \\
&=&\int\limits_{0}^{1}\int\limits_{0}^{1}\frac{L(T\left\vert t^{\ast
}-s^{\ast }\right\vert )}{T^{\alpha }\left\vert t^{\ast }-s^{\ast
}\right\vert ^{\alpha }}dt^{\ast }ds^{\ast }=\frac{L(T)}{T^{\alpha }}%
\int\limits_{0}^{1}\int\limits_{0}^{1}\frac{L(T\left\vert t^{\ast }-s^{\ast
}\right\vert )}{L(T)}\frac{1}{\left\vert t^{\ast }-s^{\ast }\right\vert
^{\alpha }}dt^{\ast }ds^{\ast }  \notag \\
&\sim &B_{0}(T)\int\limits_{0}^{1}\int\limits_{0}^{1}\frac{dt^{\ast
}ds^{\ast }}{\left\vert t^{\ast }-s^{\ast }\right\vert ^{\alpha }}=\frac{2}{%
(1-\alpha )(2-\alpha )}B_{0}(T),  \notag
\end{eqnarray}%
\noindent where we have applied Lemma \ref{lem5} with $f(t,s)=\left\vert
t^{\ast }-s^{\ast }\right\vert ^{-\alpha }.$ For $\eta >0$ one can take any
number such that $\alpha +\eta <1.$ From (\ref{A.5}), we therefore obtain
that $I_{1}(u)\longrightarrow 0,$ as $T\longrightarrow \infty ,$ in
probability, pointwise, for $u\in V^{c}(R).$

On the other hand,
\begin{eqnarray}
P\left\{ \sup_{\left\Vert u_{1}-u_{2}\right\Vert \leq h}\left\vert
I_{1}(u_{1})-I_{1}(u_{2})\right\vert >r\right\} &\leq &r^{-1}\mathit{E}%
\sup_{\left\Vert u_{1}-u_{2}\right\Vert \leq h}\left\vert
\int\limits_{0}^{T}\varepsilon (t)\frac{H_{i}(t;u_{1},u_{2})}{d_{iT}(\theta )%
}dt\right\vert  \label{A.7} \\
&\leq &r^{-1}\sup_{\left\Vert u_{1}-u_{2}\right\Vert \leq h}\sup_{t\in \left[
0,T\right] }\frac{|H_{i}(t;u_{1},u_{2})|}{d_{iT}(\theta )}E\left\vert
\varepsilon (0)\right\vert T.  \notag
\end{eqnarray}%
Under \textbf{B3}, we have
\begin{equation}
\sup_{\left\Vert u_{1}-u_{2}\right\Vert \leq h}\sup_{t\in \left[ 0,T\right] }%
\frac{|H_{i}(t;u_{1},u_{2})|}{d_{iT}(\theta )}\leq h\sup_{t\in \left[ 0,T%
\right] }\left\{ \sum_{l=1}^{q}\sup_{u\in V^{C}(R)}\frac{\left\vert
h_{il}(t;u)\right\vert }{d_{il,T}(\theta )}\cdot \frac{d_{il,T}(\theta )}{%
d_{iT}(\theta )d_{lT}(\theta )}\right\} \leq h\left( \sum_{l=1}^{q}k_{il}(R)%
\tilde{k}_{il}\right) T^{-1}.  \label{A.8}
\end{equation}

From (\ref{A.8}), we obtain
\begin{equation}
P\left\{ \sup_{\left\Vert u_{1}-u_{2}\right\Vert \leq h}\left\vert
I_{1}(u_{1})-I_{1}(u_{2})\right\vert >r\right\} \leq k_{1}r^{-1}h,
\label{A.9}
\end{equation}%
where%
\begin{equation*}
k_{1}=\left( \sum_{l=1}^{q}k_{il}(R)\tilde{k}_{il}\right) E\left\vert
\varepsilon (0)\right\vert .
\end{equation*}

Let $N_{h}$ be a finite $h-$net of the ball $V^{c}(R)$. Then
\begin{equation}
\sup_{u\in V^{c}(R)}\left\vert I_{1}(u)\right\vert \leq \sup_{\left\Vert
u_{1}-u_{2}\right\Vert \leq h}\left\vert
I_{1}(u_{1})-I_{1}(u_{2})\right\vert +\max_{u\in N_{h}}\left\vert
I_{1}(u)\right\vert .  \label{A.10}
\end{equation}

From (\ref{A.7}) and (\ref{A.8}), and for any $r>0,$%
\begin{equation}
P\left\{ \sup_{u\in V^{c}(R)}\left\vert I_{1}(u)\right\vert >r\right\} \leq
2k_{1}r^{-1}h+P\left\{ \max_{u\in N_{h}}\left\vert I_{1}(u)\right\vert >%
\frac{r}{2}\right\} .  \label{A.11}
\end{equation}%
For $\epsilon >0,$ we have $h=\frac{\epsilon r}{4k_{1}}.$ Since $I_{1}(u)%
\overset{P}{\longrightarrow }0$ pointwise, for $\ T>T_{0},$
\begin{equation*}
P\left\{ \max_{u\in N_{\frac{\epsilon r}{4k_{1}}}}\left\vert
I_{1}(u)\right\vert >\frac{r}{2}\right\} \leq \frac{\epsilon }{2},
\end{equation*}%
and
\begin{equation*}
P\left\{ \sup_{u\in V^{c}(R)}\left\vert I_{1}(u)\right\vert >r\right\} \leq
\epsilon .
\end{equation*}%
Thus, $I_{1}(u)\longrightarrow 0$, as $T\longrightarrow \infty ,$ in
probability, uniformly for $u\in V^{c}(R).$

From \textbf{B3} and Cauchy-Schwartz inequality, applying the Lagrange
formula%
\begin{eqnarray}
\sup_{u\in V^{c}(R)}\sup_{t\in \left[ 0,T\right] }\left\vert
H(t;0,u)\right\vert &=&\sup_{u\in V^{c}(R)}\sup_{t\in \left[ 0,T\right]
}\left\vert \sum_{i=1}^{q}\frac{h_{i}(t;u_{t}^{\ast })}{d_{iT}(\theta )}%
u_{i}\right\vert  \label{A.12} \\
&\leq &\sup_{u\in V^{c}(R)}\left\Vert u\right\Vert \left[ \sup_{t\in \left[
0,T\right] }\sum_{i=1}^{q}\left( \frac{h_{i}(t;u_{t}^{\ast })}{d_{iT}(\theta
)}\right) ^{2}\right] ^{1/2}\leq \left\Vert k(R)\right\Vert RT^{-1/2},
\notag
\end{eqnarray}%
where $k(R)=(k^{1}(R),\ldots ,k^{q}(R)).$

From (\ref{A.2}) and (\ref{A.12}), we have
\begin{eqnarray*}
\sup_{u\in V^{c}(R)}\left\vert I_{2}(u)\right\vert &=&\sup_{u\in
V^{c}(R)}\left\vert \int\limits_{0}^{T}H(t;0,u)\frac{H_{i}(t;u,0)}{%
d_{iT}(\theta )}dt\right\vert \\
&\leq &T\sup_{u\in V^{c}(R)}\sup_{t\in \left[ 0,T\right] }\left\vert H(t;0,u)%
\frac{H_{i}(t;u,0)}{d_{iT}(\theta )}\right\vert \leq \left\Vert
k(R)\right\Vert R^{2}T^{-1/2}\left( \sum_{l=1}^{q}k_{il}(R)\tilde{k}%
_{il}\right) .
\end{eqnarray*}%
Thus, $I_{2}(u)\longrightarrow 0$, as $T\longrightarrow \infty ,$ uniformly
for $u\in V^{c}(R)$.

Now, $I_{3}(u)$ can be written as
\begin{eqnarray*}
I_{3}(u) &=&\int\limits_{0}^{T}\frac{g_{i}(t,\theta )}{d_{iT}(\theta )}\left[
H(t;0,u)+\sum_{l=1}^{q}\frac{g_{l}(t,\theta )}{d_{lT}(\theta )}u_{l}\right]
dt=-\frac{1}{2}\int\limits_{0}^{T}\frac{g_{i}(t,\theta )}{d_{iT}(\theta )}%
\sum_{l,j=1}^{q}\frac{h_{lj}(t;u_{T}^{\ast })}{d_{lT}(\theta )d_{jT}(\theta )%
}u_{l}u_{j}dt \\
&=&-\frac{1}{2}\sum_{l,j=1}^{q}\left( \int\limits_{0}^{T}\frac{%
h_{lj}(t;u_{T}^{\ast })}{d_{lT}(\theta )d_{jT}(\theta )}\frac{g_{i}(t,\theta
)}{d_{iT}(\theta )}dt\right) u_{l}u_{j},\quad u_{T}^{\ast }\in V^{c}(R).
\end{eqnarray*}

From \textbf{B3}, applying Cauchy-Schwartz inequality:
\begin{equation*}
\sup_{u\in V^{c}(R)}\left\vert I_{3}(u)\right\vert \leq \frac{T}{2}%
k^{i}(R)\left( \sum_{l,j=1}^{q}k_{jl}(R)\tilde{k}_{jl}\left\vert
u_{j}\right\vert \ \left\vert u_{l}\right\vert \right) T^{-3/2}\leq \frac{%
qk^{i}(R)}{2}\max_{1,\leq j,l\leq q}\left\{ k_{jl}(R)\tilde{k}_{jl}\right\}
\left\Vert u\right\Vert ^{2}T^{-1/2}.
\end{equation*}

Thus, $I_{3}(u)\longrightarrow 0$, as $T\longrightarrow \infty ,$ uniformly
for $u\in V^{c}(R)$. Theorem 2 then follows.

\subsection*{Appendix 2}

The proof of the Theorem 3 is derived in this appendix.

We consider a Hessian $\mathcal{H}_{T}(w)=\left( \mathcal{H}%
_{T}^{il}(w)\right) _{i,l=1}^{q},$%
\begin{eqnarray}
\mathcal{H}_{T}^{il}(w) &=&\frac{\partial ^{2}}{\partial w_{i}\partial w_{l}}%
\left( \frac{1}{2T}Q_{T}(\theta +T^{1/2}d_{T}^{-1}(\theta )w)\right)
\label{B.1} \\
&=&T^{-1}\int\limits_{0}^{T}\left( \left[ x(t)-f(t,w)\right] \frac{%
-f_{il}(t,w)}{d_{iT}(\theta )d_{lT}(\theta )}T+\frac{f_{i}(t,w)f_{l}(t,w)}{%
d_{iT}(\theta )d_{lT}(\theta )}T\right) dt  \notag \\
&=&\int\limits_{0}^{T}\left[ f(t,0)+\varepsilon (t)-f(t,w)\right] \frac{%
-f_{il}(t,w)}{d_{iT}(\theta )d_{lT}(\theta )}dt+\int\limits_{0}^{T}\frac{%
f_{i}(t,w)f_{l}(t,w)}{d_{iT}(\theta )d_{lT}(\theta )}dt  \notag \\
&=&\int\limits_{0}^{T}F(t;w,0)\frac{f_{il}(t,w)}{d_{iT}(\theta
)d_{lT}(\theta )}dt-\int\limits_{0}^{T}\varepsilon (t)\frac{f_{il}(t,w)}{%
d_{iT}(\theta )d_{lT}(\theta )}dt  \notag \\
&&+\int\limits_{0}^{T}\frac{%
(f_{i}(t,w)-f_{i}(t,0)+f_{i}(t,0))(f_{l}(t,w)-f_{l}(t,0)+f_{l}(t,0))}{%
d_{iT}(\theta )d_{lT}(\theta )}dt  \notag \\
&=&I_{1}^{il}(w)+I_{2}^{il}(w)+\int\limits_{0}^{T}\frac{%
(f_{i}(t,w)-f_{i}(t,0))(f_{l}(t,w)-f_{l}(t,0))}{d_{iT}(\theta )d_{lT}(\theta
)}  \notag \\
&&+\int\limits_{0}^{T}\frac{(f_{i}(t,w)-f_{i}(t,0))f_{l}(t,0)}{d_{iT}(\theta
)d_{lT}(\theta )}dt+\int\limits_{0}^{T}\frac{%
f_{i}(t,0))(f_{l}(t,w)-f_{l}(t,0))}{d_{iT}(\theta )d_{lT}(\theta )}%
dt+\int\limits_{0}^{T}\frac{f_{i}(t,0))f_{l}(t,0)}{d_{iT}(\theta
)d_{lT}(\theta )}dt  \notag \\
&=&I_{1}^{il}(w)+I_{2}^{il}(w)+I_{3}^{il}(w)+I_{4}^{il}(w)+I_{5}^{il}(w)+J_{il,T}(\theta ),\quad i,l=1,\ldots ,q.
\notag
\end{eqnarray}

From the inequality%
\begin{equation}
\left\vert \lambda _{\min }(\mathcal{H}_{T}(w))-\lambda _{\min
}(J_{T}(\theta ^{0}))\right\vert \leq q\max_{1\leq i,l\leq q}\left\vert
\mathcal{H}_{T}^{il}(w)-J_{T}^{il}(\theta )\right\vert  \label{B.2}
\end{equation}%
(see, Wilkinson (1965), p.103), we have%
\begin{equation}
\max_{1\leq i,l\leq q}\left\vert \mathcal{H}_{T}^{il}(w)-J_{T}^{il}(\theta
)\right\vert \leq \sum_{m=1}^{5}\max_{1\leq i,l\leq q}\left\vert
I_{m}^{il}(w)\right\vert .  \label{B.3}
\end{equation}%
Applying Cauchy-Schwartz inequality, and the Lagrange formula, for $%
\left\Vert w\right\Vert \leq r_{0}:$%
\begin{eqnarray}
\left\vert I_{1}(w)\right\vert &=&\left\vert \int\limits_{0}^{T}F(t;w,0)%
\frac{f_{il}(t,w)}{d_{iT}(\theta )d_{lT}(\theta )}dt\right\vert \leq
T\sup_{t\in \left[ 0,T\right] }\left\vert F(t;w,0)\frac{f_{il}(t,w)}{%
d_{iT}(\theta )d_{lT}(\theta )}\right\vert  \label{B.4} \\
&\leq &T\ \hat{k}^{il}\ \tilde{k}^{il}T^{-1}\sup_{t\in \left[ 0,T\right]
}\left\vert F(t;w,0)\right\vert \leq \hat{k}^{il}\ \tilde{k}^{il}\sup_{t\in %
\left[ 0,T\right] }\left\vert f(t,w)-f(t,0)\right\vert  \notag \\
&=&\hat{k}^{il}\ \tilde{k}^{il}\sup_{t\in \left[ 0,T\right] }\left\vert
T^{1/2}\sum_{l=1}^{q}\frac{f_{l}(t,w_{T}^{\ast })}{d_{lT}(\theta )}%
w_{l}\right\vert  \notag \\
&\leq &\hat{k}^{il}\ \tilde{k}^{il}T^{1/2}\sup_{t\in \left[ 0,T\right]
}\left( \sum_{l=1}^{q}\left( \frac{f_{l}(t,w_{T}^{\ast })}{d_{lT}(\theta )}%
\right) ^{2}\right) ^{1/2}\left\Vert w\right\Vert \leq \left\Vert \hat{k}%
\right\Vert \hat{k}^{il}\ \tilde{k}^{il}\left\Vert w\right\Vert ,\ \hat{k}=(%
\hat{k}_{1},\ldots ,\hat{k}_{q}).  \notag
\end{eqnarray}

We now consider
\begin{equation}
\left\vert I_{2}(w)\right\vert =\left\vert \int\limits_{0}^{T}\varepsilon (t)%
\frac{f_{il}(t,w)}{d_{iT}(\theta )d_{lT}(\theta )}dt\right\vert =\left\vert
\int\limits_{0}^{T}\varepsilon (t)\frac{F_{il}(t;w,0)}{d_{iT}(\theta
)d_{lT}(\theta )}dt+\int\limits_{0}^{T}\varepsilon (t)\frac{f_{il}(t,0)}{%
d_{iT}(\theta )d_{lT}(\theta )}dt\right\vert \leq \left\vert
I_{6}(w)\right\vert +\left\vert I_{7}(w)\right\vert .  \label{B.5}
\end{equation}

From \textbf{B5}, we have%
\begin{equation}
\left\vert I_{6}(w)\right\vert =\left\vert \int\limits_{0}^{T}\varepsilon (t)%
\frac{F_{il}(t;w,0)}{d_{iT}(\theta )d_{lT}(\theta )}dt\right\vert \leq
\left( \frac{1}{T}\int\limits_{0}^{T}\varepsilon ^{2}(t)\right) ^{1/2}\left(
T\frac{\Phi _{il,T}(w,0)}{d_{iT}^{2}(\theta )d_{lT}^{2}(\theta )}\right)
^{1/2}\leq \left( \frac{1}{T}\int\limits_{0}^{T}\varepsilon ^{2}(t)\right)
^{1/2}(\hat{k}_{il})^{1/2}\left\Vert w\right\Vert .  \label{B.6}
\end{equation}

Then, $\left( \frac{1}{T}\int\limits_{0}^{T}\varepsilon ^{2}(t)\right)
^{1/2}=\left( \frac{1}{T}\int\limits_{0}^{T}\left( \varepsilon ^{2}(t)-%
\mathit{E}\varepsilon ^{2}(0))dt+\mathit{E}\varepsilon ^{2}(0\right) \right)
^{1/2}=\left( \xi _{T}+\mathit{E}\varepsilon ^{2}(0)\right) ^{1/2},$
\begin{equation*}
\mathit{E}\xi _{T}^{2}=T^{-2}\int\limits_{0}^{T}\int\limits_{0}^{T}\left(
\mathit{E}\varepsilon ^{2}(t)\varepsilon ^{2}(s)-(\mathit{E}\varepsilon
^{2}(0))^{2}\right) dtds.
\end{equation*}%
We will prove that
\begin{equation}
T^{-2}\int\limits_{0}^{T}\int\limits_{0}^{T}\mathit{E}\varepsilon
^{2}(t)\varepsilon ^{2}(s)dtds\longrightarrow (\mathit{E}\varepsilon
^{2}(0))^{2},\ T\longrightarrow \infty .  \label{B.7}
\end{equation}%
Under \textbf{A2}, the function $G^{2}(x)\in L_{2}(\mathbb{R}^{1},\phi
(x)dx) $ , $G^{2}(x)=\sum_{k=0}^{\infty }\frac{d_{k}}{k!}H_{k}(x),$ $%
d_{k}=\int\limits_{-\infty }^{\infty }G^{2}(x)H_{k}(x)\phi (x)dx,\quad k\geq
0.$ Thus,
\begin{equation}
\mathit{E}\varepsilon ^{2}(t)\varepsilon ^{2}(s)-(\mathit{E}\varepsilon
^{2}(0))^{2}=\sum_{k=0}^{\infty }\frac{d_{k}^{2}}{k!}B^{k}(t-s)-(\mathit{E}%
\varepsilon ^{2}(0))^{2}=\sum_{k=1}^{\infty }\frac{d_{k}^{2}}{k!}%
B^{k}(t-s)\leq \left\vert B(t-s)\right\vert \sum_{k=1}^{\infty }\frac{%
d_{k}^{2}}{k!}.  \label{B.8}
\end{equation}

Since%
\begin{equation}
\sum_{k=1}^{\infty }\frac{d_{k}^{2}}{k!}=\mathit{E}G^{4}(\xi (0))-\left[
\mathit{E}G^{2}(\xi (0))\right] ^{2}=D\mathit{E}G^{2}(\xi (0))<\infty ,
\label{B.9}
\end{equation}%
we have from (\ref{B.8})-(\ref{B.9}) as $T\longrightarrow \infty ,$%
\begin{equation*}
T^{-2}\int\limits_{0}^{T}\int\limits_{0}^{T}\left( \mathit{E}\varepsilon
^{2}(t)\varepsilon ^{2}(s)-(\mathit{E}\varepsilon ^{2}(0))^{2}\right)
dtds\leq D\mathit{E}G^{2}(\xi
(0))T^{-2}\int\limits_{0}^{T}\int\limits_{0}^{T}\left\vert B(t-s)\right\vert
dtds\longrightarrow 0,
\end{equation*}%
as it was already proven in Appendix 1.

We will write
\begin{equation}
\xi _{T}=o_{p}^{(1)}(1),\quad o_{p}^{(1)}(1)\longrightarrow ^{P}0,\quad %
\mbox{as}\ T\longrightarrow \infty .  \label{B.10}
\end{equation}

From \textbf{B3} and \textbf{B5},
\begin{eqnarray*}
\mathit{E}\left\vert I_{7}(w)\right\vert ^{2} &=&\mathit{E}\left\vert
\int\limits_{0}^{T}\varepsilon (t)\frac{f_{il}(t;0)}{d_{iT}(\theta
)d_{lT}(\theta )}dt\right\vert ^{2}\leq
\int\limits_{0}^{T}\int\limits_{0}^{T}\left\vert \mathit{E}\varepsilon
(t)\varepsilon (s)\right\vert \left( \sup_{t\in \left[ 0,T\right] }\frac{%
f_{il}(t;0)}{d_{iT}(\theta )d_{lT}(\theta )}\right) ^{2} \\
&\leq &\left( \hat{k}^{il}\ \tilde{k}^{il}\right) ^{2}D\mathit{E}G(\xi
(0))T^{-2}\int\limits_{0}^{T}\int\limits_{0}^{T}\left\vert B(t-s)\right\vert
^{m}dtds,
\end{eqnarray*}%
and
\begin{equation}
\left\vert I_{7}(w)\right\vert =o_{p}^{(2)}(1),\quad
o_{p}^{(2)}(1)\longrightarrow ^{P}0,\quad \mbox{as}\quad T\longrightarrow
\infty .  \label{B.11}
\end{equation}

We can continue as follow%
\begin{eqnarray}
\left\vert I_{3}(w)\right\vert &=&\left\vert \int\limits_{0}^{T}\frac{%
(f_{i}(t,w)-f_{i}(t,0))(f_{l}(t,w)-f_{l}(t,0))}{d_{iT}(\theta )d_{lT}(\theta
)}dt\right\vert  \label{B.12} \\
&\leq &T\sum_{j=1}^{q}\sum_{s=1}^{q}\int\limits_{0}^{T}\frac{\left\vert
f_{ij}(t,w_{T}^{\ast })\right\vert }{d_{iT}(\theta )d_{jT}(\theta )}\ \frac{%
\left\vert f_{ls}(t,w_{T}^{\ast })\right\vert }{d_{sT}(\theta )d_{lT}(\theta
)}dt\left\vert w_{j}\right\vert \ \left\vert w_{s}\right\vert  \notag \\
&\leq &T\sum_{j=1}^{q}\sum_{s=1}^{q}\int\limits_{0}^{T}\frac{\left\vert
f_{ij}(t,w_{T}^{\ast })\right\vert }{d_{ij,T}(\theta )}\ \frac{%
d_{ij,T}(\theta )}{d_{iT}(\theta )d_{jT}(\theta )}\ \frac{\left\vert
f_{ls}(t,w_{T}^{\ast })\right\vert }{d_{ls,T}(\theta )}\frac{d_{ls,T}(\theta
)}{d_{sT}(\theta )d_{lT}(\theta )}dt\left\vert w_{j}\right\vert \ \left\vert
w_{s}\right\vert  \notag \\
&\leq &\left( \sum_{s=1}^{q}\left( \hat{k}^{ls}\ \tilde{k}^{ls}\right)
^{2}\right) ^{1/2}.\left( \sum_{j=1}^{q}\left( \hat{k}^{ij}\ \tilde{k}%
^{ij}\right) ^{2}\right) ^{1/2}\left\Vert w\right\Vert ^{2}.  \notag
\end{eqnarray}

From \textbf{B3} and \textbf{B5}, we obtain
\begin{eqnarray}
\left\vert I_{4}(w)\right\vert &=&\left\vert \int\limits_{0}^{T}\frac{%
(f_{i}(t,w)-f_{i}(t,0))f_{l}(t,0)}{d_{iT}(\theta )d_{lT}(\theta )}%
dt\right\vert  \label{B.13} \\
&\leq &T^{1/2}\sum_{j=1}^{q}\left\vert w_{j}\right\vert \int\limits_{0}^{T}%
\frac{\left\vert f_{ij}(t,w_{T}^{\ast })\right\vert }{d_{ij,T}(\theta )}\
\frac{d_{ij,T}(\theta )}{d_{iT}(\theta )d_{jT}(\theta )}\ \frac{\left\vert
f_{l}(t,0)\right\vert }{d_{lT}(\theta )}dt  \notag \\
&\leq &T^{3/2}\sum_{j=1}^{q}\left\vert w_{j}\right\vert \sup_{t\in \left[ 0,T%
\right] }\left\{ \frac{\left\vert f_{ij}(t,w_{T}^{\ast })\right\vert }{%
d_{ij,T}(\theta )}\ \frac{d_{ij,T}(\theta )}{d_{iT}(\theta )d_{jT}(\theta )}%
\ \frac{\left\vert f_{l}(t,0)\right\vert }{d_{lT}(\theta )}\right\}  \notag
\\
&\leq &\hat{k}^{l}\left( \sum_{j=1}^{q}\left( \hat{k}^{ij}\ \tilde{k}%
^{ij}\right) ^{2}\right) ^{1/2}\left\Vert w\right\Vert .  \notag
\end{eqnarray}%
Similarly for $I_{5}(w),$ we have%
\begin{equation}
\left\vert I_{5}(w)\right\vert \leq \hat{k}^{i}\left( \sum_{j=1}^{q}\left(
\hat{k}^{lj}\ \tilde{k}^{lj}\right) ^{2}\right) ^{1/2}\left\Vert
w\right\Vert .  \label{B.14}
\end{equation}

From (\ref{B.4})-(\ref{B.14}), we get%
\begin{equation*}
\left\vert \lambda _{\min }(\mathcal{H}_{T}(w))-\lambda _{\min
}(J_{T}(\theta ^{0}))\right\vert \leq q\max_{1\leq i,l\leq q}\left\{
\left\Vert \hat{k}\right\Vert \hat{k}^{il}\ \tilde{k}^{il}\left\Vert
w\right\Vert +\left[\left( o_{p}^{(1)}(1)+1\right) ^{1/2}\left( \hat{k}%
_{ij}\right) ^{1/2}\left\Vert w\right\Vert +o_{p}^{(2)}(1)\right]\right.
\end{equation*}%
\begin{equation}
\left. +\left( \sum_{s=1}^{q}\left( \hat{k}^{ls}\ \tilde{k}^{ls}\right)
^{2}\right) ^{1/2}\left( \sum_{j=1}^{q}\left( \hat{k}^{ij}\ \tilde{k}%
^{ij}\right) ^{2}\right) ^{1/2}\left\Vert w\right\Vert ^{2}+\hat{k}%
^{l}\left( \sum_{j=1}^{q}\left( \hat{k}^{ij}\ \tilde{k}^{ij}\right)
^{2}\right) ^{1/2}\left\Vert w\right\Vert +\hat{k}^{i}\left(
\sum_{j=1}^{q}\left( \hat{k}^{lj}\ \tilde{k}^{lj}\right) ^{2}\right)
^{1/2}\left\Vert w\right\Vert \right\}.  \label{B.15}
\end{equation}

We substitute the estimate (\ref{B.15}) into the normed LSE $\hat{w}_{T}$,
and, by \textbf{B4} (from which it is follows that $J_{T}(\theta ^{0})>0,$
with the minimal eigenvalue $\lambda _{\min }(J_{T}(\theta ^{0}))\geq
\lambda _{\ast },$ for some $r>0,$ we introduce the event%
\begin{equation*}
\Omega _{1}\cap \Omega _{2}\cap \Omega _{3}=\left\{ \left\vert
o_{p}^{(1)}(1)\right\vert \leq r,\ \left\vert o_{p}^{(2)}(1)\right\vert \leq
r,\left\Vert \hat{w}_{T}\right\Vert \leq r\right\} \subset \left\{
\left\vert \lambda _{\min }(\mathcal{H}_{T}(\hat{w}_{T}))-\lambda _{\min
}(J_{T}(\theta ^{0}))\right\vert \leq \frac{\lambda _{\ast }}{2}\right\}
\end{equation*}%
\begin{eqnarray}
&=&\left\{ \lambda _{\min }(J_{T}(\theta ^{0})-\frac{\lambda _{\ast }}{2}%
\leq \lambda _{\min }(\mathcal{H}_{T}(\hat{w}_{T}))\leq \lambda _{\min
}(J_{T}(\theta ^{0})+\frac{\lambda _{\ast }}{2}\right\}  \label{B.16} \\
&\subset &\left\{ \lambda _{\min }(\mathcal{H}_{T}(\hat{w}_{T}))\geq \lambda
_{\min }(J_{T}(\theta ^{0})-\frac{\lambda _{\ast }}{2}\right\} \subset
\left\{ \lambda _{\min }(\mathcal{H}_{T}(\hat{w}_{T}))\geq \frac{\lambda
_{\ast }}{2}\right\} .  \notag
\end{eqnarray}

We then have
\begin{equation*}
P\left\{ \overline{\Omega _{1}\cap \Omega _{2}\cap \Omega _{3}}\right\} \leq
P\left\{ \left\vert o_{p}^{(1)}(1)\right\vert >r\right\} +P\left\{
\left\vert o_{p}^{(2)}(1)\right\vert >r\right\} +P\left\{ \left\Vert \hat{w}%
_{T}\right\Vert >r\right\} .
\end{equation*}

For any $\epsilon >0$ and $T>T_{0},$ we obtain $P\left\{ \left\vert
o_{p}^{(1)}(1)\right\vert >r\right\} \leq \frac{\epsilon }{3},\ P\left\{
\left\vert o_{p}^{(2)}(1)\right\vert >r\right\} \leq \frac{\epsilon }{3}.$
Note that if $T>T_{0},$ then, $P\left\{ \left\Vert \hat{w}_{T}\right\Vert
>r\right\} \leq \frac{\epsilon }{3},$ and hence,
\begin{equation*}
P\left\{ \left\vert o_{p}^{(1)}(1)\right\vert >r\right\} +P\left\{
\left\vert o_{p}^{(2)}(1)\right\vert >r\right\} +P\left\{ \left\Vert \hat{w}%
_{T}\right\Vert >r\right\} \leq \frac{\epsilon }{3}+\frac{\epsilon }{3}+%
\frac{\epsilon }{3}=\epsilon .
\end{equation*}%
Therefore, for $T>T_{0},$ $P\left\{ \Omega _{1}\cap \Omega _{2}\cap \Omega
_{3}\right\} >1-\epsilon .$ This means that the normed $\hat{w}_{T}$ is the
unique solution of the equation (\ref{5.3}), with probability tending to 1
as $T\longrightarrow \infty ,$ since the matrix $\mathcal{H}_{T}(\hat{w}%
_{T}) $ is positive definite and the functional $Q_{T}(\theta ,\omega )$ has
unique minimum at the point $\hat{w}_{T}.$ Thus, Theorem 3 is proven.


\begin{thebibliography}{99}
\bibitem{r2} \textsc{Anh, V.V.}, \textsc{Knopova, V.P.} and \textsc{%
Leonenko, N.N.} (2004). Continuous-time stochastic processes with cyclical
long-range dependence. \textit{Australian and NZ J. of Statistics} \textbf{46%
}, 275--296.

\bibitem{r3} \textsc{Arcones, M. A.} (1994). Limit theorems for nonlinear
functionals of a stationary Gaussian sequence of vectors. \textit{Ann.
Probab.} \textbf{22}, 2242--2274.

\bibitem{r4} \textsc{Arcones, M. A.} (2000). Distributional limit theorems
over a stationary Gaussian sequence of random vectors. \textit{Stochastic
Process. Appl.} \textbf{88}, 135--159.

\bibitem{r5} \textsc{Avram, F.} (1992). Generalized Szeg{\"o} Theorems and
asymptotics of cumulants by graphical methods. \textit{Transactions of the
American Mathematical Society} \textbf{330}, 637--649.

\bibitem{r6} \textsc{Avram, F.} and \textsc{Brown, L. A.} (1989).
Generalized H{\"o}lder Inequality and a Generalized Szeg{\"o} Theorem.
\textit{Proceedings of the American Mathematical Society} \textbf{107},
687--695.

\bibitem{r7} \textsc{Avram, F.} and \textsc{Fox, R.} (1992). Central limit
theorems for sums of Wick products of stationary sequences. \textit{%
Transactions of the American Mathematical Society} \textbf{330}, 651--663.

\bibitem{r8} \textsc{Avram, F.} \textsc{Leonenko N.N.} and \textsc{Sakhno, L.%
} (2010). On Szeg{\"o} type limit theorem, the Holder --Young-Brascamp-Lieb
inequality, and asymptotic theory of integrals and quadratic forms of
stationary fields. \textit{ESAIM: Probability and Statistics}, \textbf{14},
210-255.

\bibitem{Bh} \textsc{Bhattacharya, R.N. and Ranga Rao, R. }(1976). Normal
Approximation and Asymptotic Expansions, Wiley, New Yoark

\bibitem{r10} \textsc{Berman, S.M.} (1992). A central limit theorem for the
renormalized self-intersection local time of a stationary vector Gaussian
process. \textit{Annals of Probability} \textbf{20}, 61-81.

\bibitem{r11} \textsc{Breuer, P.} and \textsc{Major, P.} (1983). Central
limit theorems for nonlinear functionals of Gaussian fields. \textit{J.
Multiv. Anal.} \textbf{13}, 425--441.

\bibitem{r15} \textsc{Dobrushin R.L.} and \textsc{Major, P.} (1979).
Non-central limit theorem for non-linear functionals of Gaussian fields.
\textit{Z Wahrsch. Verw. Geb.} \textbf{50}, 1--28.

\bibitem{r16} \textsc{Donoghue, W.J.} (1969). \textit{Distributions and
Fourier Transforms.} Academic Press, New York.

\bibitem{r17} \textsc{Doukhan, P.}, \textsc{Oppenheim, G.} and \textsc{%
Taqqu, M.S.} (2003).\textit{Theory and Applications of Long-range Dependence}%
. Birkh{\"a}user, Boston.

\bibitem{GR} \textsc{Gradshteyn, I.S. and Ryzhik, I.M.}(2000) Tables of
Integrals, Series and Products, sixth ed., Academic Press, San Diego.

\bibitem{r18} \textsc{Grenander, U.} and \textsc{Rosenblatt, M.} (1984).
\textit{Statistical Analysis of Stationary Time Series}. Chelsea Publ.
Company, New York.

\bibitem{Hannan} \textsc{Hannan, E.}(1973) The asymptotic theory of linear
time-series models. \textit{J. Appl. Probab.}, \textbf{10}, 510-519

\bibitem{r19} \textsc{Haye, O.M.} (2002). Asymptotic behavior of the
empirical processes for Gaussian data presented seasonal long-memory.
\textit{ESAIM: Probability and Statistics} \textbf{6}, 293--309.

\bibitem{r20} \textsc{Haye, O.M.} and \textsc{Phillipe, A.} (2003). A
noncentral limit theorem for the empirical processes of linear sequences
with seasonal long memory. \textit{Mathematical Methods of Statistics}
\textbf{12}, 329--357.

\bibitem{r21} \textsc{Haye, O.M.} and \textsc{Viano M.-C.} (2002). \textit{%
Limit theorems under seasonal long-memory. In Theory and Applications of
Long-Range Dependence}. Doukhan, P et al. (Eds.). Birkhauser, Boston,
101--110.

\bibitem{r23} \textsc{Ho, H.C.} and \textsc{Hsing, T.} (1997). Limit
theorems for functionals of moving average. \textit{Ann.Probab.} \textbf{25}%
, 1636-1669.

\bibitem{r22} \textsc{Holevo, A.S.} (1976). On the asymptotic efficient
regression estimates in the case of degenerate spectrum. \textit{Theory
Probab. Appl.} \textbf{21}, 324--333.

\bibitem{r24} \textsc{Ibragimov, I.A.} and \textsc{Rozanov, Yu.A.} (1980).
Gaussian Random Processes. Springer-Verlag, New York.

\bibitem{r25} \textsc{Ivanov, A.V.} (1980). A solution of the problem of
detecting hidden periodicities. \textit{Theory Probability and Math.Stat.}
\textbf{20}, 51-68.

\bibitem{r26} \textsc{Ivanov, A.V.} (1997). \textit{Asymptotic Theory of
Nonlinear Regression}. Kluwer Academic Publishers. Dordrecht.

\bibitem{Iv1} \textsc{Ivanov, A.V }(2010) Consistency of the least squares
estimator of the amplitudes and angular frequencies of the sum of harmonic
oscillations in models with strong dependence. \textit{Theory Probab. Math.
Statist}. \textbf{80}, 61--69

\bibitem{r27} \textsc{Ivanov, A.V.} and \textsc{Leonenko, N.N.} (1989).
\textit{Statistical Analysis of Random Fields}. Kluwer Academic Publishers,
Dordrecht.

\bibitem{r28} \textsc{Ivanov, A.V.} and \textsc{Leonenko, N.N.} (2004).
Asymptotic theory for non-linear regression with long-range dependence.
\textit{Mathematical Methods of Statistics} \textbf{13}, 153--178.

\bibitem{r29} \textsc{Ivanov, A.V.} and \textsc{Leonenko, N.N.} (2008).
Semiparametric analysis of long-range dependence in nonlinear regression,
\textit{J. Statist. Plann. Inference} \textbf{138}, 1733-1753.

\bibitem{Iv2} \textsc{Ivanov, A.V.} and \textsc{Leonenko, N.N.} (2009)
Robust estimators in non-linear regression models with long-range
dependence. \textit{Optimal design and related areas in optimization and
statistics}, 193--221, Springer Optim. Appl., 28, Springer, New York.

\bibitem{Iv3} \textsc{Ivanov, A.V.} and \textsc{Orlovskii I.V.} (2008)
Asymptotic normality of \$M\$M-estimates in the classical nonlinear
regression model, \textit{Ukrainian Math. J}. \textbf{60}, no. 11,
1716--1739.

\bibitem{Koul1} \textsc{Koul H. }(1996) Asymptotics of M-estimations in
non-linear regression with long-range dependence errors. In: Proc. Athens
Conf. Appl. Probab. and Time Ser. Analysis (P.M. Robinson and M. Rosenblatt,
Eds.) Springer Verlag Lecture Notes in Statistics, II, 272-291.

\bibitem{Koul2} \textsc{Koul H and Baillie, R.T. }(2003) Asymptotics of
M-estimators in non-linear regression models with long-memory designs.
\textit{Statistics and Probability Letters, }\textbf{61}, 237--252.

\bibitem{r31} \textsc{Leonenko, N.N.} (1999). \textit{Limit Theorems for
Random Fields with Singular Spectrum}. Kluwer Academic Publishers, Dordrecht.

\bibitem{r32} \textsc{Leonenko, N.N.} and \textsc{Taufer, E.} (2006). Weak
convergence of weighted quadratic functionals of stationary long memory
processes to Rosenblatt-type distributions. \emph{J. Statist. Planning and
Inference} \textbf{136}, 1220--1236.

\bibitem{Miln} \textsc{Milnor, J.W. }(1965) \textit{Topology from the
Differentiable Viewpoint}. Princeton University Press. Princeton, NJ.

\bibitem{Muk} \textsc{Mukhergee, K} (2000) Linearization of randomly
weighted empiricals under long range dependence with applications to
nonlinear regression quantiles.\textit{\ Economic Theory}, \textbf{16},
301-323

\bibitem{r35} \textsc{Nualart, D.} (1995). \textit{The Malliavin Calculus
and Related Topics}. Springer-Verlag, New York.

\bibitem{r36} \textsc{Nualart, D.} and \textsc{Peccati, G.} (2005). Central
limit theorems for sequences of multiple stochastic integrals. \textit{The
Annals of Probability} \textbf{33}, 177-193.

\bibitem{r37} \textsc{Oppenheim, G.}, \textsc{Ould H.M.} and \textsc{Viano
M.-C.} (2002). Long memory with seasonal effects. \textit{Statitical
Inference for Stochastic Processes} \textbf{3}, 53-68.

\bibitem{r38} \textsc{Peccati, G.} (2009). Stein's method, Malliavin
calculus and infinite-dimensional Gaussian analysis. \emph{Lecture Notes}
(www.glocities.com/giovannipeccati).

\bibitem{r39} \textsc{Peccati, G.} and \textsc{Taqqu, M.S.} (2010). \textit{%
Wiener Chaos: Moments, Cumulants and Diagrams}. Springer, Berlin.

\bibitem{r40} \textsc{Peccati, G.} and \textsc{Tudor, C.A.} (2004). Gaussian
limits for vector-valued multiple stochastic integrals. \textit{S{\'e}%
minaire de Probabilit{\'e}s XXXVIII}, 247--262.

\bibitem{Pol} \textsc{Pollard, D. and Radchenko, P. (2006)} Nonlinear
least-squares estimation. \textit{J.}{\ }\textit{Multiv. Anal., }\textbf{97}%
, 548-562.

\bibitem{rQuinn} \textsc{Quinn, B.G. and Hannan, E.J.} (2001). \textit{The
Estimation and Tracking of Frequency}. Cambridge University Press. New York.

\bibitem{rob1} \textsc{Robinson, P.M. and Hidalgo, F.J.} (1997) Time series
regression with long-range dependence. \textit{Ann. Statist.,} \textbf{25},
77-104

\bibitem{r41} \textsc{Rosenblatt, M.} (1976). Fractional integrals of
stationary processes and the central limit theorem. \textit{J. Appl. Prob.}
\textbf{13}, 723--732.

\bibitem{r42} \textsc{Rosenblatt, M.} (1981). Limit theorems for Fourier
transforms of functionals of Gaussian sequences. \textit{Z. Wahrsch.
 Verw. Gebiete } \textbf{55}, 123-132.

\bibitem{r43} \textsc{Rosenblatt, M.} (1987). Scale renormalization and
random solutions of Burgers equation. \textit{J. Appl. Prob.} \textbf{24},
328-338.

\bibitem{sen} \textsc{Seneta, E }(1976) Regularly Varying Functions. Lecture
Notes in Mathematics. \textbf{508}, Springer-Verlag, Berlin-New York.

\bibitem{sko} \textsc{Skouras, K }(2000) Strong consistency in nonlinear
regression models. \textit{Ann. Statist.,} \textbf{28}, 871-879

\bibitem{r46} \textsc{Taqqu, M.S.} (1975). Weak convergence to fractional
Brownian motion and to the Rosenblatt process. \textit{Z. Wahrsch.
Verw. Gebiete} \textbf{31}, 287--302.

\bibitem{r47} \textsc{Taqqu M.S.} (1979). Convergence of integrated
processes of arbitrary Hermite rank. \textit{Z. Wahrsch. Verw.
Gebiete} \textbf{50}, 53--83.

\bibitem{r50} \textsc{Viano, M.-C.}, \textsc{Deniau, Cl.} and \textsc{%
Oppenheim, G.} (1995). Long-range dependence and mixing for discrete time
fractional processes. \emph{J. Time Series Analysis} \textbf{16}, 323-338.

\bibitem{WAL} \textsc{Walker, A.M. }(1973) On the estimation of a harmonic
component in a time series with stationary dependent residuals, \textit{Adv.
Appl. Probab}., \textbf{5}, 217-241.

\bibitem{r51} \textsc{Whittle, P.} (1952). The simultaneous estimation of a
time series harmonic components and covariance structure. \textit{Trabajos
Estad{\'\i}stica} \textbf{3}, 43--57.

\bibitem{Wil} \textsc{Wilkinson, J.H. }(1965). The Algebraic Eigenvalue
Problem. Clarendon Press, Oxford.

\bibitem{r52} \textsc{Yajima, Y.} (1988). On estimation of a regression
model with long-memory stationary errors. \textit{Ann. Statist.} \textbf{16}%
, 791--807.

\bibitem{r53} \textsc{Yajima, Y.} (1991). Asymptotic properties of the LSE
in a regression model with long-memory stationary errors. \textit{Ann.
Statist.} \textbf{19}, 158--177.
\end{thebibliography}
\end{document}